\documentclass[12pt]{gtpart}

\usepackage{pinlabel}
\usepackage[all]{xy}

\title{Pixton's double ramification cycle relations}

\author[E.~Clader]{Emily Clader}
\givenname{Emily}
\surname{Clader}
\email{eclader@sfsu.edu}
\address{Department of Mathematics\\ San Francisco State University\\ \newline San Francisco, CA, 94132\\ United States}
\urladdr{https://sites.google.com/site/emilyclader/}

\author[F.~Janda]{Felix Janda}
\givenname{Felix}
\surname{Janda}
\email{janda@umich.edu}
\address{Department of Mathematics\\ University of Michigan\\ \newline Ann Arbor, MI 48109\\ United States}
\urladdr{http://www-personal.umich.edu/~janda/}

\keyword{Moduli of curves}
\keyword{Tautological relations}
\subject{primary}{msc2010}{14H10}
\subject{secondary}{msc2010}{14N35}

\thanks{The first author acknowledges the generous support of Dr.~Max R\"ossler, the Walter Haefner Foundation, and the ETH Foundation.  The second author was partially supported by the Swiss National Science Foundation grant SNF 200021\_143274.}

\theoremstyle{plain}
 \newtheorem{theorem}{Theorem}[section]
\theoremstyle{plain}
 
\theoremstyle{plain}
 \newtheorem{lemma}[theorem]{Lemma}
\theoremstyle{plain}
 
\theoremstyle{plain}
 \newtheorem{conjecture}[theorem]{Conjecture}
\theoremstyle{definition}
 
\theoremstyle{definition}
 
\theoremstyle{definition}
 \newtheorem{remark}[theorem]{Remark}
\theoremstyle{definition}

\usepackage{mathrsfs}
\usepackage{enumerate}

\renewcommand{\Z}{\ensuremath{\mathbb{Z}}}

\renewcommand{\C}{\ensuremath{\mathbb{C}}}
\renewcommand{\P}{\ensuremath{\mathbb{P}}}
\newcommand{\M}{\ensuremath{\overline{\mathcal{M}}}}
\renewcommand{\O}{\ensuremath{\mathcal{O}}}
\newcommand{\ev}{\ensuremath{\textrm{ev}}}
\DeclareMathOperator{\Aut}{Aut}

\newcommand{\vir}{\ensuremath{\textrm{vir}}}
\renewcommand{\t}{\ensuremath{\mathbf{t}}}
\newcommand{\ch}{\ensuremath{\textrm{ch}}}
\newcommand{\CZr}{\ensuremath{[\mathbb{C}/\mathbb{Z}_r]}}
\newcommand{\DR}[2]{\ensuremath{R_{#1,#2}}}
\newcommand{\OmegaCZr}{\ensuremath{C^{r}}}
\newcommand{\RCZr}{\ensuremath{R^{r}_C}}

\newcommand{\Conj}[2]{\ensuremath{\Omega_{#1,#2}}}
\newcommand{\Conjr}[2]{\ensuremath{\Omega^{r}_{#1,#2}}}
\newcommand{\Conjrgg}[2]{\ensuremath{\Omega^{r\gg0}_{#1,#2}}}
\newcommand{\Zvr}[2]{\ensuremath{\widetilde{\Omega}^{r}_{#1,#2}}}
\newcommand{\Zvrgg}[2]{\ensuremath{\widetilde{\Omega}^{r\gg0}_{#1,#2}}}
\newcommand{\OmegaZvr}{\ensuremath{\widetilde{\Omega}^{r}}}
\newcommand{\RZvr}{\ensuremath{\widetilde{R}^{r}}}

\newcommand{\Conjk}[3]{\ensuremath{\Omega_{#1,#2,#3}}}
\newcommand{\one}{\ensuremath{\mathbf{1}}}
\newcommand{\Ch}{\ensuremath{\textrm{Ch}}}
\newcommand{\G}{\ensuremath{G}}

\begin{document}

\begin{abstract}
We prove a conjecture of Pixton, namely that his proposed formula for the double ramification cycle on $\M_{g,n}$ vanishes in codimension beyond $g$.  This yields a collection of tautological relations in the Chow ring of $\M_{g,n}$.  We describe, furthermore, how these relations can be obtained from Pixton's $3$-spin relations via localization on the moduli space of stable maps to an orbifold projective line.
\end{abstract}

\maketitle

%%%%%%%%%%%%%%%%%%%%   Start of main body of article

\section{Introduction}

The double ramification cycle is a class $\DR{g}{A} \in A^g(\M_{g,n})$ associated to any genus $g \geq 0$ and any collection of integers $A = (a_1, \ldots, a_n)$ whose sum is zero.  Its restriction to the moduli space $\mathcal{M}_{g,n} \subset \M_{g,n}$ of smooth curves is the class of the locus of pointed curves $(C; x_1, \ldots, x_n)$ admitting a ramified cover $f: C \rightarrow \P^1$, for which the positive $a_i$ describe the ramification profile over $0$ and the negative $a_i$ describe the ramification profile over $\infty$.  This definition can be extended to all of $\M_{g,n}$ via relative Gromov--Witten theory.

The question known as ``Eliashberg's problem'' is, vaguely, whether one can give a more explicit description of the double ramification cycle.  Toward this end, Faber and Pandharipande \cite{FP} proved that $\DR{g}{A}$ lies in the tautological ring, so Eliashberg's problem can be refined by asking for a formula in terms of kappa and psi classes and their pushforwards from boundary strata.

In \cite{Hain}, Hain provided such a formula for the restriction of $\DR{g}{A}$ to the compact-type locus $\mathcal{M}_{g,n}^{ct} \subset \M_{g,n}$, which parameterizes curves whose dual graph is a tree.  His proof relies on an alternative description of the double ramification cycle in terms of the universal Jacobian.  Namely, on a smooth curve $C$, the existence of a ramified cover as prescribed by the definition of $\DR{g}{A}$ is equivalent to the requirement that
\[\O_C(a_1[x_1] + \cdots + a_n[x_n]) \cong \O_C.\]
Thus, if
\[\rho_A : \mathcal{M}_{g,n} \rightarrow \mathcal{X}_g\]
is the map to the universal abelian variety defined by
\[(C;x_1, \ldots, x_n) \mapsto \O_C(a_1[x_1] + \cdots + a_n[x_n]) \in \text{Jac}^0_C\]
and $\mathcal{Z}_g \subset \mathcal{X}_g$ is the zero section, then
\begin{equation}
\label{smooth}
\DR{g}{A}|_{\mathcal{M}_{g,n}} =\rho_A^*[\mathcal{Z}_g].
\end{equation}
The map $\rho_A$ extends without indeterminacy to $\mathcal{M}_{g,n}^{ct}$, and Marcus and Wise \cite{MW}, generalizing a previous result of Cavalieri-Marcus-Wise \cite{CMW} for rational-tails curves, proved that the analogue of (\ref{smooth}) still holds on the compact-type locus.  On $\mathcal{X}_g$, there is a theta divisor $\Theta$ satisfying
\[\Theta^g = g![\mathcal{Z}_g].\]
Thus, we have
\[\DR{g}{A}|_{\mathcal{M}^{ct}_{g,n}} = \frac{1}{g!} (\rho_A^*\Theta)^g,\]
and Hain's formula results from an explicit calculation of $\rho_A^*\Theta$ in terms of tautological classes.

On the other hand, Grushevsky and Zakharov leveraged this same computation of $\rho_A^*\Theta$ in a different way \cite{GZ}.  Namely, they used the observation that
\[\Theta^{g+1} = 0\]
to derive tautological relations in $A^d(\mathcal{M}_{g,n}^{ct})$ for any $d>g$.

In recent work \cite{PixtonDRNotes} (see also \cite{CavalieriDR}), Pixton defined an extension of Hain's class to the entire moduli space $\M_{g,n}$.  More precisely, he extended the mixed-degree class $e^{\rho_A^*\Theta}$ to a more general formula in terms of tautological classes, denoted $\Conj{g}{A}$.  
To construct it, he first defined a family of classes $\Conjr{g}{A}$ depending on a positive integer parameter $r$, which can in some sense be viewed as ``mod $r$'' versions of Hain's expression for $R_{g,A}|_{\mathcal{M}^{ct}_{g,n}}$.  He then proved that $\Conjr{g}{A}$ is polynomial in $r$ for $r \gg 0$, and he defined $\Conj{g}{A}$ as the constant term in this polynomial.

Simultaneously generalizing both Hain's and Grushevsky-Zakharov's arguments, Pixton conjectured the following:

\begin{conjecture}[Pixton]
	\label{conj}
	Let $[ \cdot]_d$ denote the codimension-$d$ part of a class in $A^*(\M_{g,n})$.  Then $\Conj{g}{A}$ satisfies:
	\begin{enumerate}
		\item $[\Conj{g}{A}]_{g} = \DR{g}{A}$;
		\item $[\Conj{g}{A}]_{d} = 0$ for all $d > g$.
	\end{enumerate}
\end{conjecture}

Part (1) has recently been proven by
Janda-Pandharipande-Pixton-Zvonkine \cite{JPPZ}, using localization on
a moduli space of relative stable maps to an orbifold projective line.
In particular, since $\Conj{g}{A}$ has an explicit expression in terms
of the additive generators of the tautological ring, this yields a
solution to Eliashberg's problem.

The main result of the present paper is a proof of part (2):

\begin{theorem}
\label{main}
Let $\Conj{g}{A} \in A^*(\M_{g,n})$ be the mixed-degree class defined by (\ref{DRconj}), whose codimension-$g$ component is equal to the double ramification cycle.  Then the component of $\Conj{g}{A}$ in codimension $d$ vanishes for all $d > g$.
\end{theorem}

To prove the theorem, we make use of a geometric reformulation of $\Conj{g}{A}$ due to Zvonkine.  Namely, we consider a moduli space $\M^{0/r}_{g,A}$ of pointed stable curves $(C;x_1, \ldots, x_n)$ equipped with a line bundle $L$ satisfying
\[L^{\otimes r} \cong \O\left(-\sum_{i=1}^n a_i[x_i]\right).\]
There is a map
\[\phi: \M^{0/r}_{g,A} \rightarrow \M_{g,n}\]
forgetting the line bundle $L$, and if $\pi: \mathcal{C} \rightarrow \M^{0/r}_{g,A}$ denotes the universal curve and $\mathcal{L}_A$ the universal line bundle, set
\begin{equation}
\label{Zvr}
\Zvr{g}{A}:= \frac{1}{r^{2g-1}}\phi_*\left(e^{r^2c_1(-R\pi_*\mathcal{L}_A)}\right).
\end{equation}
Like Pixton's class, $\Zvr{g}{A}$ is also polynomial in $r$ for $r\gg 0$, and the constant term in this polynomial is also equal to $\Conj{g}{A}$.

From here, the idea of the proof of Theorem~\ref{main} is to replace $A$ by a tuple $A'$ in such a way that $-R\pi_*\mathcal{L}_{A'}$ becomes a vector bundle but the constant term in $r$ of (\ref{Zvr}) remains unchanged.  Then, we replace the class $e^{r^2c_1(-R\pi_*\mathcal{L}_{A'})}$ with the weighted total Chern class
\begin{equation}
\label{total}
c_{(r^2)}(-R\pi_*\mathcal{L}_{A'}) = 1 + r^2c_1(-R\pi_*\mathcal{L}_{A'}) + r^4 c_2(-R\pi_*\mathcal{L}_{A'}) + \cdots.
\end{equation}
Once again, this replacement only affects higher-order terms in $r$; the proof uses the fact that both $\Zvr{g}{A}$ and the modification via (\ref{total}) form Cohomological Field Theories (CohFTs), and the $R$-matrices can be explicitly calculated by Chiodo's Grothendieck-Riemann-Roch formula \cite{ChiodoTowards}.  The rank of $-R\pi_*\mathcal{L}_{A'}$ is easy to compute, and for certain choices of $A$, the modification $A'$ can be chosen so that this rank equals precisely $g$.  For such $A$, the fact that (\ref{total}) manifestly vanishes in cohomological degrees greater than the rank proves the theorem.  Then, using the fact that $\Conj{g}{A}$ is polynomial in $A$ (as observed by Pixton \cite{PixtonForthcoming}), we deduce the theorem in general.

\begin{remark}
The tautological relations coming from vanishing of the high-degree terms of (\ref{total}) were previously observed in \cite{Clader}.  As was explained in that paper, they can alternatively be derived from the existence of the nonequivariant limit in the equivariant virtual cycle of $\M_{g,n}(\CZr, 0)$, a perspective that is useful in what follows.
\end{remark}

Pixton also conjectured that the same vanishing result holds for a more general class, which we denote by $\Conjk{g}{A}{k}$.  This class can also be described by a Hain-type formula, as we explain in Section \ref{powers}, or, in the geometric reformulation, it can be defined by considering the class
\begin{equation}
\label{eq:LAk}
\frac{1}{r^{2g-1}} \phi_*(e^{r^2c_1(-R\pi_*\mathcal{L}_{A,k})}),
\end{equation}
where $\mathcal{L}_{A,k}$ is the universal line bundle over the moduli space $\M_{g,A}^{k/r}$ of pointed stable curves with a line bundle $L$ satisfying
\begin{equation}
\label{Lrk}
L^{\otimes r} \cong \omega_{\log}^{\otimes k} \left(-\sum_{i=1}^n a_i[x_i]\right).
\end{equation}
Once again, \eqref{eq:LAk} is polynomial in $r$ for $r \gg 0$, and $\Omega_{g,A,k}$ is defined as the constant term in this polynomial.  The $k=1$ case, in particular, is related to $r$-spin theory. We prove in Theorem~\ref{thm2} that Pixton's conjecture for $\Omega_{g,A,k}$ is also true, by essentially the same proof as Theorem~\ref{main}.

\begin{remark}
This more general vanishing is connected to relations studied by Randal-Williams \cite{RW, ERW}, building on ideas of Morita \cite{Morita1, Morita2}.  Specifically, Randal-Williams works on the $n$th fiber product $\mathcal{C}_g^n$ of the universal curve over $\mathcal{M}_g$, and his relations are the image of the restriction of the relations in Theorem~\ref{thm2} under the birational morphism $\mathcal{M}_{g,n}^{rt} \rightarrow \mathcal{C}_g^n$, where $\mathcal{M}_{g,n}^{rt}$ is the moduli space of rational-tails curves.
\end{remark}

It has been conjectured that the $3$-spin relations constructed in
\cite{PPZ} generate all tautological relations on the moduli space of
curves, so one should expect the double ramification cycle relations
of Theorem~\ref{main} to follow from these.  This is indeed the case:

\begin{theorem}
\label{main2}
The double ramification cycle relations are a consequence of Pixton's $3$-spin relations.
\end{theorem}

To prove Theorem~\ref{main2}, we study the equivariant Gromov--Witten theory of a
projective line $\P[r,1]$ with a single orbifold point of isotropy
$\Z_r$.  The associated CohFT is generically semisimple, so, as
explained in \cite{Janda}, tautological relations can be obtained by
applying Givental-Teleman reconstruction to express the CohFT as a
graph sum and then using the existence of the limit as one moves
toward a nonsemisimple point.  The relations thus obtained are
equivalent to the $3$-spin relations, via rather general machinery of
the second author.

On the other hand, the same CohFT can be expressed as a graph sum in a different way, via localization and Chiodo's formula.  A careful matching reveals that the two graph sums agree, and the existence of the nonsemisimple limit in the Givental-Teleman sum implies the existence of the nonequivariant limit in the localization sum.  Thus, upon restriction to the substack of degree-zero maps to $\P[r,1]$, one recovers the double ramification cycle relations in the form presented in \cite{Clader}.

\subsection{Outline of the paper}

We begin, in Section \ref{preliminaries}, by reviewing the definition of the double ramification cycle and Pixton's conjectural formula in more detail.  In Section \ref{GRR}, we recall Chiodo's Grothendieck-Riemann-Roch formula for the Chern characters of the direct image of the universal line bundle on moduli spaces of $r$th roots and use it to make the formula for $\Zvr{g}{A}$ more explicit.  Section \ref{proof} reduces the proof of Theorem~\ref{main} to a comparison of $\Zvr{g}{A}$ with the weighted total Chern class described in (\ref{total}), and this comparison is accomplished in Section \ref{Rmatrices} by describing both classes in terms of the action of an explicit $R$-matrix on a Topological Field Theory, thus completing the proof of the main theorem and its generalization.  Finally, in Section \ref{3spin}, we recast the double ramification cycle relations in terms of maps to an orbifold projective line, and use this perspective to show how they can be deduced from the $3$-spin relations.  Details of the localization on $\P[r,1]$, including a matching of the localization and reconstruction graph sums, are relegated to the Appendix.

\subsection{Acknowledgments}

The authors are especially indebted to R. Pandharipande, A. Pixton, Y. Ruan, and D. Zvonkine for numerous invaluable conversations and guidance.
Thanks are also due to A. Chiodo, J. Gu\'er\'e, T. Milanov and D. Ross for useful conversations and comments.  Much of the authors' initial understanding of the double ramification cycle was shaped by lectures of R. Cavalieri given at the workshop ``Modern Trends in Gromov--Witten Theory'' at the Universit\"at Hanover, as well as by lectures of A. Pixton at the ETH Z\"urich, both of which occurred in September 2014.

\section{Preliminaries on the double ramification cycle and Pixton's conjectures}
\label{preliminaries}

The exposition that follows is based on notes of Cavalieri \cite{CavalieriDR} and Pixton \cite{PixtonDRNotes}.

\subsection{The double ramification cycle}

Fix a genus $g \geq 0$ and a collection of integers $A = (a_1, \ldots, a_n)$ whose sum is zero.  Define a cycle on $\mathcal{M}_{g,n}$ as the class of the locus of pointed curves $(C;x_1, \ldots, x_n)$ for which there exists a ramified cover $f: C \rightarrow \P^1$ satisfying:
\begin{itemize}
\item $f^{-1}(0) = \{x_i \; | \; a_i > 0\}$,
\item the ramification profile over $0$ is the partition $\{a_i \; | a_i > 0\}$,
\item $f^{-1}(\infty) = \{x_i \; | \; a_i < 0\}$,
\item the ramification profile over $\infty$ is the partition $\{|a_i| \; | \; a_i < 0\}$.
\end{itemize}
We denote by $\mu$ the partition consisting of the positive $a_i$ and by $\nu$ the partition consisting of the absolute values of the negative $a_i$; these are partitions of the same size since the sum of all $a_i$'s is zero.  Further, denote $n_0 = \#\{a_i = 0\}$; note that no restriction is placed on the $x_i$ for which $a_i=0$.

To extend the class described above to the entire moduli space $\M_{g,n}$, we compactify the space of such ramified covers by allowing degenerations of the target $\P^1$.  More specifically, there is a map
\[\pi: \M_{g,n_0}(\P^1; \mu, \nu)^{\sim} \rightarrow \M_{g,n}\]
from the moduli space of rubber relative stable maps to $\P^1$, and we set
\[\DR{g}{A} := \pi_*[\M_{g,n_0}(\P^1;\mu, \nu)^{\sim}]^{\vir} \in A^g(\M_{g,n}).\]
See \cite{FP} for a further discussion of rubber relative stable maps to the projective line.

This class has an alternative description when restricted to the locus $\mathcal{M}_{g,n}^{ct} \subset \M_{g,n}$ consisting of curves of compact type--- that is, curves whose dual graph is a tree.  As explained in the introduction, the Jacobian $\text{Jac}^0_C$ of a compact-type curve is a (compact) abelian variety, and the map
\[\rho_A: \mathcal{M}_{g,n} \rightarrow \mathcal{X}_g\]
to the universal abelian variety defined by
\[(C;x_1, \ldots, x_n) \mapsto \O_C(a_1[x_1] + \cdots + a_n[x_n]) \in \text{Jac}^0_C\]
can be extended to $\mathcal{M}_{g,n}^{ct}$.  It is straightforward to see that, if $\mathcal{Z}_g \subset \mathcal{X}_g$ denotes the zero section, then the class $\rho_A^*[\mathcal{Z}_g]$ coincides with the double ramification cycle when one restricts to $\mathcal{M}_{g,n}$.  By the results of \cite{CMW} and \cite{MW}, this is also true for the extension to $\mathcal{M}_{g,n}^{ct}$.

On the other hand, there is a theta divisor $\Theta \in A^1(\mathcal{X}_g)$, which restricts in each fiber of the universal family to the prescribed polarization on the corresponding abelian variety, and which is trivial when restricted to the zero section.  Using results of Deninger and Murre \cite{DM} (see \cite{Voisin} and \cite{GH} for further exposition), one can show that this divisor satisfies
\[\Theta^g = g! [\mathcal{Z}_g]\]
and $\Theta^{g+1} = 0$.

Hain \cite{Hain} computed $\rho_A^*\Theta$ in terms of tautological classes on $\mathcal{M}_{g,n}^{ct}$, which, via the above observations, implies a formula for the restriction of the double ramification cycle.  The result of his computation is:
\begin{equation}
\label{Hains}
\DR{g}{A}^{ct}  = \frac{1}{2^g g!} \left( -\frac{1}{2} \sum_{\substack{ 0 \leq l \leq g \\ I\subset \{1, \ldots, n\}}} a_I^2 \Delta_{l,I} \right)^g,
\end{equation}
where
\[a_I = \sum_{i \in I} a_i\]
and $\Delta_{l,I}$ is defined as the class of the closure of the locus of curves with an irreducible component of genus $l$ containing the marked points in $S$ and an irreducible component of genus $g-l$ containing the remaining marked points.  (In the unstable cases where such curves do not exist, it is defined by convention: $\Delta_{0, \{i\}} = \Delta_{g, [n] \setminus \{i\}}= -\psi_i$, and $\Delta_{0, \emptyset} = 0$.)

\subsection{Pixton's conjectural formula}
\label{conjecture}

The starting point for Pixton's generalization of Hain's formula (\ref{Hains}) to all of $\M_{g,n}$ is the observation that, by packaging the expressions for each power of $\rho_A^*\Theta$ into the mixed-degree class $e^{\rho_A^*\Theta}$, one obtains a ``compact-type Cohomological Field Theory''.  That is, if $V$ is an infinite-dimensional vector space with generators $e_a$ indexed by integers $a$, then the association
\[V^{\otimes n} \rightarrow H^*(\mathcal{M}_{g,n}^{ct})\]
\[e_{a_1} \otimes \cdots \otimes e_{a_n} \mapsto \DR{g}{A}^{ct}\]
satisfies all of the axioms of a CohFT {\it except} for the gluing axiom along nonseparating nodes, which do not occur in the compact-type moduli space.  We refer the reader to \cite{KM} or \cite{PPZ} for a careful discussion of CohFTs and their axioms.

According to the results of Givental and Teleman \cite{GiventalSemisimple, Teleman}, a semisimple CohFT can be obtained via the action of an $R$-matrix on a Topological Field Theory; the result is an expression for the CohFT as a summation over graphs.  A similar procedure works for $\DR{g}{A}^{ct}$, and it can be used to write Hain's formula as a graph sum.  Namely, by expanding the exponential and using intersection theory on $\M_{g,n}$, one finds that
\begin{equation}
\label{Pixtonct}
e^{\rho_A^*\Theta}=  \sum_{\Gamma \in G^{ct}_{g,n}}\frac{\iota_{\Gamma*}}{|\Aut(\Gamma)|} \left(\prod_{i=1}^n e^{\frac{1}{2}a_i^2\psi_i} \hspace{-0.5cm} \prod_{e=(h,h') \in E(\Gamma)} \hspace{-0.4cm}\frac{1 - e^{-\frac{1}{2}w(h)w(h')(\psi_h + \psi_{h'})}}{\psi_h + \psi_{h'}}\right).
\end{equation}
Here, $G^{ct}_{g,n}$ denotes the set of decorated dual graphs of curves in $\mathcal{M}_{g,n}^{ct}$.  The set of edges of a graph $\Gamma$ is denoted $E(\Gamma)$, and each edge is written $e = (h,h')$ for half-edges $h$ and $h'$.  The classes $\psi_h$ and $\psi_{h'}$ are the first Chern classes of the cotangent line bundles at the two branches of the node corresponding to $e$, and $\iota_{\Gamma}$ is the gluing map
\begin{equation}
\label{iotagamma}
\iota_{\Gamma}: \prod_{\text{vertices }v} \M_{g(v),\text{val}(v)} \rightarrow \M_{g,n},
\end{equation}
in which $g(v)$ is the genus of $v$ and $\text{val}(v)$ the valence (that is, the total number of half-edges and legs incident to $v$).  

Associated to each such graph $\Gamma$ is a unique weight function
\[w: H(\Gamma) \rightarrow \Z\]
on the set $H(\Gamma)$ of half-edges and legs, determined by:
\begin{enumerate}[(\bf{{W}}1)]
\item $w(h_i) = a_i$ for each leg $h_i$ associated to a marked point $x_i$;
\item if $e = (h,h')$, then $w(h) + w(h') = 0$;
\item for each vertex $v$, the sum of the weights of half-edges and legs incident to $v$ equals zero.
\end{enumerate}
The fact that these conditions uniquely determine $w$ is a consequence of the fact that $\Gamma$ is a tree.

Now, if one attempts to na\"ively generalize the above formula to the full moduli space by allowing $\Gamma$ to be any dual graph for a curve in $\M_{g,n}$, then there will no longer be a unique choice of weight function $w$ satisfying {\bf (W1) -- (W3)}.  Indeed, any loop in the dual graph permits infinitely many choices of weights, so the sum of the expressions in (\ref{Pixtonct}) over all possible weight functions will not converge.

To avoid such infinite sums, Pixton introduces an additional parameter $r$ and restricts to weight functions
\[w: H(\Gamma) \rightarrow \{0, 1, \ldots, r-1\}\]
satisfying the following three conditions:
\begin{enumerate}[(\bf{{R}}1)]
\item $w(h_i) \equiv a_i \mod r$ for each half-edge $h_i$ associated to a marked point $x_i$;
\item if $e = (h,h')$, then $w(h) + w(h') \equiv 0 \mod r$;
\item for each vertex $v$, the sum of the weights of half-edges incident to $v$ is zero modulo $r$.
\end{enumerate}
There are clearly only finitely many such weight functions associated to any dual graph $\Gamma$.  Set $\Conjr{g}{A}$ to be the class
\[\sum_{\Gamma,w} \frac{1}{|\Aut(\Gamma)|} \frac{1}{r^{h_1(\Gamma)}}\iota_{\Gamma*}\left(\prod_{i=1}^n e^{\frac{1}{2}a_i^2\psi_i} \prod_{e=(h,h')} \frac{1 - e^{-\frac{1}{2}w(h)w(h')(\psi_h + \psi_{h'})}}{\psi_h + \psi_{h'}}\right),\]
where $\Gamma$ ranges over all dual graphs of curves in $\M_{g,n}$, and $w$ ranges over weight functions satisfying {\bf (R1) -- (R3)}.

As observed by Pixton, the class $\Omega_{g,A}^r$ satisfies a number of polynomiality properties:

\begin{lemma}[Pixton \cite{PixtonForthcoming}]
  \label{polynomiality}
  For fixed $g$ and $A$, the class $\Omega_{g,A}^r$ is polynomial in
  $r$ for $r \gg 0$.  Moreover, the constant term in this polynomial
  is itself polynomial in the arguments $A$.

  More generally, let $\Gamma$ be a dual graph with half-edges $h_1,
  \dotsc, h_N$ and let $W$ be a polynomial in $N$ variables.  Then the
  sum
  \begin{equation*}
    \sum_w W(w(h_1), \dotsc, w(h_N)),
  \end{equation*}
  where $w$ ranges over weight functions satisfying {\bf (R1) --
    (R3)}, is a polynomial in $r$ for $r \gg 0$.  This polynomial is divisible by $r^{h_1(\Gamma)}$ and its lowest-degree term depends on $a_1, \dotsc, a_n$ polynomially.

\end{lemma}

Given Lemma \ref{polynomiality}, Pixton's conjectural formula for the double ramification cycle can now be defined:
\begin{equation}
\label{DRconj}
\Conj{g}{A} := \Conjrgg{g}{A}\bigg|_{r=0},
\end{equation}
where $\Conjrgg{g}{A}$ is defined as the class $\Conjr{g}{A}$ for any $r$ large enough so that this class is polynomial in $r$.

\subsection{Geometric reformulation}
\label{geom}

A different perspective on $\Conj{g}{A}$, first suggested by Zvonkine, will be more useful for our proof of Theorem~\ref{main}.  Let $\M^{0/r}_{g,A}$ be the moduli space\footnote{Here, a compactification of the moduli space of such objects on smooth curves must be chosen.  There are several ways to compactify, as summarized in Section 1.1.2 of \cite{SSZ}; for our purposes, we will allow orbifold structure at the nodes of $C$ and require only that $L$ is an orbifold line bundle.} parameterizing pointed stable curves $(C; x_1, \ldots, x_n)$ equipped with a line bundle $L$ satisfying
\begin{equation}
\label{Lr}
L^{\otimes r}  \cong \O\left(-\sum_{i=1}^n a_i[x_i]\right).
\end{equation}

There is a map
\[\phi: \M^{0/r}_{g,A} \rightarrow \M_{g,n}\]
forgetting the line bundle $L$ and the orbifold structure; this map has degree $r^{2g-1}$, as explained, for example, in \cite{Chiodo}.  If $\pi: \mathcal{C}_A \rightarrow \M^{0/r}_{g,A}$ denotes the universal curve and $\mathcal{L}_A$ denotes the universal line bundle on $\mathcal{C}_A$, then the class
\begin{equation}
\label{ZvrgA}
\Zvr{g}{A}: = \frac{1}{r^{2g-1}} \phi_*(e^{-r^2 c_1(R\pi_{*}\mathcal{L}_A)})
\end{equation}
is also polynomial in $r$ for $r \gg 0$, and
\[\Conj{g}{A} = \Zvrgg{g}{A}\bigg|_{r=0}.\]
The fact that this definition of $\Conj{g}{A}$ agrees with the previous one (and that $\Zvr{g}{A}$, like $\Conjr{g}{A}$, is eventually polynomial in $r$) can be proved by noting that $\Zvr{g}{A}$ forms a semisimple CohFT on a vector space $V = \C\{e_0, e_1, \ldots, e_{r-1}\}$, expressing it as a dual graph sum using the Givental-Teleman reconstruction of semisimple CohFTs, and comparing the resulting dual graph sums using Lemma~\ref{polynomiality}.  We return to this argument in Lemma~\ref{PixtonZvonkine} below.

\subsection{Generalization to powers of the log canonical}
\label{powers}

Both of these definitions of $\Omega_{g,A}$ are readily generalized to allow for powers of the log canonical.  To do so, fix an integer $k$ and assume that $A = (a_1, \ldots, a_n)$ satisfies
\[\sum_{i=1}^n a_i  = k(2g-2+n).\]
Let $\M_{g,A}^{k/r}$ be the moduli space parameterizing pointed stable curves $(C;x_1, \ldots, x_n)$ equipped with a line bundle $L$ satisfying
\begin{equation}
\label{Lk}
L^{\otimes r} \cong \omega_{\log}^{\otimes k} \left(-\sum_{i=1}^n a_i[x_i]\right).
\end{equation}
As above, there is a degree-$r^{2g-1}$ map
\[\phi: \M_{g,A}^{k/r} \rightarrow \M_{g,n}\]
forgetting $L$ and the orbifold structure on the curve.  Set
\[\widetilde{\Omega}^r_{g,A,k}:= \frac{1}{r^{2g-1}} \phi_*\left(e^{-r^2c_1(R\pi_*\mathcal{L}_{A,k})}\right),\]
where $\pi: \mathcal{C}_{A,k} \rightarrow \M_{g,A}^{k/r}$ is the universal curve and $\mathcal{L}_{A,k}$ the universal line bundle.

A generalization of Pixton's class can be defined by
\[\Omega_{g,A,k} = \widetilde{\Omega}^{r\gg0}_{g,A,k}\bigg|_{r=0},\]
in the language of Section \ref{geom}.  Alternatively, in Pixton's original formulation, the generalized class is defined by replacing condition {\bf(R3)} above by
\begin{enumerate}[(\bf{{R}}3$'$)]
\setcounter{enumi}{2}
\item for each vertex $v$, the sum of the weights of half-edges incident to $v$ is $k(2g(v) - 2 + \text{val}(v))$ modulo $r$
\end{enumerate}
and setting $\Omega^r_{g,A,k}$ to be the class
\begin{multline*}
\sum_{\Gamma,w} \frac{1}{|\Aut(\Gamma)|} \frac{1}{r^{h_1(\Gamma)}}\cdot\\
\iota_{\Gamma*}\left(\prod_v e^{-\frac{1}{2}k^2\kappa_1} \prod_{i=1}^n e^{\frac{1}{2}a_i^2\psi_i} \prod_{e=(h,h')} \frac{1 - e^{-\frac{1}{2}w(h)w(h')(\psi_h + \psi_{h'})}}{\psi_h + \psi_{h'}}\right),
\end{multline*}
where $\Gamma$ ranges over all dual graphs of curves in $\M_{g,n}$, $v$ ranges over vertices of $\Gamma$, and $w$ ranges over weight functions satisfying {\bf (R1), (R2)}, and {\bf (R3$'$)}.  Pixton has also proven an analogue of Lemma~\ref{polynomiality} for $\Omega^r_{g,A,k}$:
\begin{lemma}[Pixton \cite{PixtonForthcoming}]
  \label{kpolynomiality}
  For fixed $g$ and $A$, the class $\Omega^r_{g,A,k}$ is polynomial in
  $r$ for $r \gg 0$.  Moreover, the constant term in this polynomial
  is itself polynomial in $k$ and the arguments $A$.

  More generally, let $\Gamma$ be a dual graph with half-edges $h_1,
  \dotsc, h_N$ and let $W$ be a polynomial in $N$ variables.  Then the
  sum
  \begin{equation*}
    \sum_w W(w(h_1), \dotsc, w(h_N)),
  \end{equation*}
  where $w$ ranges over weight functions satisfying {\bf (R1), (R2)},
  and {\bf (R3$'$)}, is a polynomial in $r$ for $r \gg 0$ and is divisible by $r^{h_1(\Gamma)}$.
\end{lemma}
We can therefore define $\Omega_{g,A,k}$ as the constant term of the
polynomial in $r$ corresponding to $\Omega^r_{g,A,k}$.  When $k=0$, we recover the previous definitions of $\Omega_{g,A}$.  Until otherwise stated, we will always assume that $k=0$ in what follows.

\section{Chiodo's Grothendieck--Riemann--Roch formula}
\label{GRR}

In this section, we recall Chiodo's formula for the Chern characters of the direct image $R\pi_*\mathcal{L}_A$, which, in particular, can be used to write (\ref{ZvrgA}) explicitly in terms of tautological classes when $r$ is sufficiently large.

Fix a tuple of integers $A =(a_1, \ldots, a_n)$.  In fact, one need not assume that the sum of the $a_i$ is zero, as was the case above, but only that
\[\sum_{i=1}^n a_i \equiv 0 \mod r;\]
this more general version will be important later.  Let $\pi$ and $\mathcal{L}_A$ be as above.  Then Chiodo's formula states:
\begin{align*}
&\ch_d(R\pi_{*}\mathcal{L}_A) = \frac{B_{d+1}(0)}{(d+1)!}\kappa_d - \sum_{i=1}^n \frac{B_{d+1}(\frac{a_i}{r})}{(d+1)!}\psi_i^d \\
&\hspace{0.75cm}+\frac{r}{2} \sum_{\substack{0 \leq l \leq g\\ I \subset [n]}} \frac{B_{d+1}\left(\frac{q_{l,I}}{r}\right)}{(d+1)!} i_{(l,I)*}(\gamma_{d-1}) + \frac{r}{2} \sum_{q=0}^{r-1} \frac{B_{d+1}(\frac{q}{r})}{(d+1)!}j_{(irr,q)*}(\gamma_{d-1}),
\end{align*}
using the presentation given in Corollary 3.1.8 of \cite{ChiodoTowards}.

Let us summarize the notation appearing in this formula.  First, $B_{d+1}(x)$ are the Bernoulli polynomials, defined by the generating function
\[\frac{te^{xt}}{e^t -1} = \sum_{n=0}^{\infty} B_n(x) \frac{t^n}{n!}.\]
The $\kappa$ and $\psi$ classes are defined as usual, using the cotangent line to the coarse underlying curve.

Let $Z_{(l,I)}$ be the substack of $\mathcal{C}_A$ consisting of nodes separating the curve $C$ into a component of genus $l$ containing the marked points in $I$ and a component of genus $g-l$ containing the other marked points, subject to the requirement that stable curves of this type exist.  Let $Z_{(l,I)}'$ be the two-fold cover of $Z_{(l,I)}$ given by a choice of branch at each such node.  Then
\[i_{(l,I)}: Z'_{(l, I)} \rightarrow \M^{0/r}_{g,A}\]
is the composition of this two-fold cover with the inclusion into $\mathcal{C}_A$ and projection.  The index $q_{l,I} \in \{0,1, \ldots, r-1\}$ is the multiplicity of $L$ at the chosen branch, which is defined by
\[q_{l,I} + \sum_{i \in I} a_i \equiv 0 \mod r.\]
If $\psi$ is the first Chern class of the line bundle over $Z'_{(l,I)}$ whose fiber is the cotangent line to the coarse curve at the chosen branch of the node, and $\hat{\psi}$ is the first Chern class of the bundle whose fiber is the cotangent line to the coarse curve at the opposite branch, then $\gamma_d$ is defined by
\begin{equation}
\label{gamma}
\gamma_d  = \frac{\psi^{d+1} + (-1)^d \hat\psi^{d+1}}{\psi + \hat\psi}= \sum_{i+j=d} (-\psi)^i \hat{\psi}^j.
\end{equation}

Finally, let $Z'_{(\text{irr},q)}$ be given by nonseparating nodes in $\mathcal{C}_A$ together with a choice of branch, such that the multiplicity of the line bundle $L$ at the chosen branch is equal to $q$.  We have morphisms
\[j_{(\text{irr},q)}: Z'_{(irr,q)} \rightarrow \M^{0/r}_{g,A}\]
given, as before, by the two-fold cover, inclusion into the universal curve, and projection.  The class $\gamma_d$ is again defined by (\ref{gamma}).

\section{Comparison with the total Chern class}
\label{proof}

Fix a collection of integers $A = (a_1, \ldots, a_n)$ whose sum is zero.  Suppose that $n > 0$ and exactly one $a_i$ is negative; without loss of generality, we may assume that $a_1 < 0$ and $a_i \geq 0$ for all $i \geq 2$.  Choose any $r > \text{max}\{|a_i|\}$, and set
\begin{equation}
\label{A'}
A' = (a_1', \dotsc, a_n') = (a_1 + r, a_2, \ldots, a_n),
\end{equation}
which is now a collection whose sum is $r$ and for which every element is nonnegative.

The definitions of the moduli space $\M^{0/r}_{g,A}$ and the class $\Zvr{g}{A}$ extend verbatim to tuples of integers whose sum is not necessarily zero but merely zero modulo $r$.  In particular, $\Zvr{g}{A'}$ is defined, and in fact, its constant term in $r$ is the same as that of $\Zvr{g}{A}$:

\begin{lemma}
\label{claim1}
If $A$ and $A'$ are as above, then
\[\left.\Zvrgg{g}{A}\right|_{r=0} = \left.\Zvrgg{g}{A'}\right|_{r=0}.\]
\begin{proof}
Via Chiodo's formula, $\Zvr{g}{A}$ can be written as
\[ \frac{1}{r^{2g-1}}\phi_*\left(\exp\left(-r^2\frac{B_2(0)}{2}\kappa_1 + r^2\sum_{i=1}^n\frac{B_2(\frac{a_i}{r})}{2} \psi_i - r^3\sum_{\Gamma} \frac{B_2(\frac{q_{\Gamma}}{r})}{2} [\Gamma]\right)\right),\]
where the sum is over one-noded graphs $\Gamma$ decorated with a multiplicity $q_{\Gamma}$ at the node, and $[\Gamma]$ is the corresponding boundary divisor.  (Note that since $B_2(x) = B_2(1-x)$, we need not distinguish between the two choices of branch.)

Because $B_2$ is a degree-two polynomial, the replacement $A \mapsto A'$ only affects the higher-order terms in $r$ in the argument of $\phi_*$.  Some care is required to ensure that the same is true after applying $\phi_*$, since the degree of $\phi$ on a codimension-$d$ boundary stratum is in general equal to $r^{2g-1-d}$ due to the presence of ``ghost'' automorphisms.  This is indeed the case, though, because the replacement of $A$ by $A'$ does not change the boundary term.
\end{proof}
\end{lemma}

We have thus re-expressed Pixton's conjectural formula (for $A$ satisfying the above conditions) as
\[\Conj{g}{A} =\left.\Zvrgg{g}{A'}\right|_{r=0}.\]
The advantage of having replaced $A$ by $A'$ is the following:

\begin{lemma}
\label{vbundle}
If $A$ and $A'$ are as above, then
\[R^0\pi_*\mathcal{L}_{A'} = 0,\]
and hence $-R\pi_*\mathcal{L}_{A'}$ is a vector bundle.
\begin{proof}
Let $(C;x_1, \ldots, x_n;L)$ be an element of $\M^{0/r}_{g,A'}$, so
\[L^{\otimes r} \cong \O\left(-\sum_{i=1}^n a_i'[x_i]\right).\]
Let $s$ be a section of $L$, and suppose that there exists an irreducible component $C'$ of $C$ on which $s \not \equiv 0$.

Since
\[\deg(L|_{C'}) = -\frac{1}{r}\sum_{i \in C'} a_i' \leq 0\]
as an orbifold line bundle, we must have $a_i' = 0$ for all $i \in C'$.  Moreover, $L$ cannot have nontrivial orbifold structure at any of the nodes of $C'$, since $s$ would necessarily vanish at such a node and hence would be identically zero on $C'$.  It follows that $L|_{C'}$ is pulled back from a degree-zero bundle on the coarse underlying curve $|C'|$.  Indeed, this bundle must be $\O_{|C'|}$, for otherwise $L|_{C'}$ would have no nonzero section.

We conclude that $s$ is nowhere zero on $C'$, and in particular, that it does not vanish at any of the nodes at which $C'$ meets the rest of $C$.  Thus, none of the components meeting $C'$ can contain a marked point $x_i$ for which $a_i' \neq 0$.  Continuing inductively, we find that $a_i' = 0$ for all $i$, a contradiction.
\end{proof}
\end{lemma}

By Lemma \ref{vbundle}, the weighted total Chern class
\[c_{(r^2)}(-R\pi_*\mathcal{L}_{A'}) := 1 + r^2c_1(-R\pi_*\mathcal{L}_{A'})  + r^{4}c_2(-R\pi_*\mathcal{L}_{A'}) + \cdots \]
is well-defined.  It can be expressed in terms of Chern characters as
\[c_{(r^2)}(-R\pi_*\mathcal{L}_{A'}) = \exp\left(\sum_{d \geq 1} (-r^2)^d (d-1)! \text{ch}_d(R\pi_*\mathcal{L}_{A'})\right),\]
and hence, it also admits an explicit description via Chiodo's formula.  Let
\[C^{r}_{g,A'} : = \frac{1}{r^{2g-1}} \phi_*\big(c_{(r^2)}(-R\pi_*\mathcal{L}_{A'})\big).\]

\begin{lemma}
\label{claim2}
One has
\[ \left.\Zvrgg{g}{A'}\right|_{r=0} = \left. C^{r\gg0}_{g,A'} \right|_{r=0}.\]
\end{lemma}

A proof of this lemma will imply Theorem~\ref{main} for the tuples $A$ under consideration, since $C^{r}_{g,A'}$ clearly vanishes past the rank of the bundle $-R\pi^*\mathcal{L}_{A'}$ and a straightforward Riemann--Roch computation shows that
\[\text{rank}(-R\pi_*\mathcal{L}_{A'}) = g-1 + \frac{1}{r}\sum_{i=1}^n a_i' = g.\]

The proof of Lemma~\ref{claim2} follows the same lines as that of Lemma~\ref{claim1}.  However, to make the argument carefully, one must be vigilant about the boundary terms appearing in both classes.  The most streamlined way to handle these is to realize that both $\Zvr{g}{A'}$ and $C^{r}_{g,A'}$ can be encoded as semisimple CohFTs, and hence can be expressed as the result of an $R$-matrix action on a Topological Field Theory (TFT).  The two classes are then compared by explicitly computing both the $R$-matrix and the TFT in each case.  This is the content of the following section.

\section{The CohFTs and their $R$-matrices}
\label{Rmatrices}

The results of this section are well-known to experts--- in particular, closely-related computations appear in \cite{CZ}, \cite{SSZ}, and \cite{CR}--- but we recall them here for clarity.

\subsection{The CohFTs}
\label{Raction}

Recall that a CohFT, as originally defined by Kontsevich and Manin \cite{KM}, consists of a finite-dimensional $\C$-vector space $V$ equipped with a nondegenerate pairing $\eta$, a distinguished element $\mathbf{1} \in V$, and a system of homomorphisms
\[\Omega_{g,n}: V^{\otimes n } \rightarrow H^*(\M_{g,n})\]
satisfying a number of compatibility axioms.  Any CohFT yields a quantum product $\ast$ on $V$, defined by
\[\eta(v_1 \ast v_2, v_3) = \Omega_{0,3}(v_1 \otimes v_2 \otimes v_3),\]
and we say that the CohFT is semisimple if $\ast$ makes $V$ into a semisimple $\C$-algebra--- that is, if there exists a basis $\epsilon_1, \ldots, \epsilon_r$ for $V$ for which
\[\epsilon_i \ast \epsilon_j = \delta_{ij} \epsilon_i.\]
The work of Givental and Teleman \cite{Givental, Teleman} implies that a semisimple CohFT can be expressed as
\[\Omega = R \cdot \omega,\]
where
\[R = R(z) \in \text{End}(V)[\![z]\!]\]
is an $R$-matrix and $\omega$ is the Topological Field Theory obtained by projecting $\Omega$ to $H^0(\M_{g,n})$.

For the reader's convenience, we briefly recall the definition of the action of an $R$-matrix on a CohFT; more detailed information can be found in \cite{PPZ}.  We have:
\[R \cdot \omega := \sum_{\Gamma \in G_{g,n}} \frac{1}{|\Aut(\Gamma)|} \text{Cont}_{\Gamma},\]
where $G_{g,n}$ is the set of decorated dual graphs of curves in $\M_{g,n}$, and $\text{Cont}_{\Gamma} \in H^*(\M_{g,n}) \otimes (V^*)^{\otimes n}$ is defined via contraction of tensors as follows:
\begin{itemize}
\item at each vertex of $\Gamma$, place the tensor
\[(T\omega)_{g(v), \text{val}(v)} \in H^*(\M_{g(v), \text{val}(v)}) \otimes (V^*)^{\otimes \text{val}(v)}\]
described below;
\item at each leg $l$ of $\Gamma$ attached to a vertex $v$, place
\[R^{-1}(\psi_l) \in H^*(\M_{g(v),\text{val}(v)}) \otimes \text{End}(V);\]
\item at each edge $e=(h,h')$ of $\Gamma$ joining vertices $v$ and $v'$, place
\[\frac{\eta^{-1} - R^{-1}(\psi_h) \eta^{-1} R^{-1}(\psi_h')^t}{\psi_h + \psi_{h'}} \in H^*(\M_{g(v), \text{val}(v)}) \otimes H^*(\M_{g(v'), \text{val}(v')}) \otimes V^{\otimes 2}.\]
\end{itemize}
In the vertex contribution, the translation operator $T$ is defined by
\[T(z) := z \one - z R^{-1}(z) \one \in z^2 V[[z]],\]
and $(T\omega)_{g,n}(v_1 \otimes \cdots \otimes v_n)$ is
\[\sum_{m \geq 0} \frac{1}{m!} p_{m*}(v_1 \otimes \cdots \otimes v_n \otimes T(\psi_{n+1}) \otimes \cdots \otimes T(\psi_{n+m})),\]
where $p_m: \M_{g,n+m} \rightarrow \M_{g,n}$ is the forgetful map.  

In our case, the underlying vector space is
\[V = \C\{\zeta_0, \zeta_1, \ldots, \zeta_{r-1}\}\]
with the pairing
\[\eta(\zeta_i, \zeta_j) = \begin{cases} 1 & \text{ if } i+j\equiv 0 \mod r\\ 0 & \text{ otherwise}.\end{cases}\]
We define two CohFTs on this vector space.

The first CohFT is
\[\OmegaZvr_{g,n}(\zeta_{a_1} \otimes \cdots \otimes \zeta_{a_n}) = r^{g} \cdot \Zvr{g}{A} = \frac{1}{r^{g-1}}\phi_*\left(e^{r^2c_1(-R\pi_*\mathcal{L}_A)}\right),\]
where $A = (a_1, \ldots, a_n)$.  The second is
\[\OmegaCZr_{g,n}(\zeta_{a_1} \otimes \cdots \otimes \zeta_{a_n}) = r^{g} \cdot C^{r}_{g,A} = \frac{1}{r^{g-1}}\phi_*\big(c_{(r^2)}(-R\pi_*\mathcal{L}_{A})\big).\]
In both cases, the class is set to zero when the moduli space $\M^{0/r}_{g,A}$ does not exist--- that is, whenever the condition
\[\sum_{i=1}^n a_i \equiv 0 \mod r\]
is not satisfied.

\begin{remark}
\label{GWtheory}
We remark that $\OmegaCZr_{g,n}$ has another interpretation, as discussed in \cite{Clader}.  Namely, we consider the orbifold $\CZr$, on which $\C^*$ acts by multiplication.  Then, if $\lambda$ denotes the equivariant parameter, one has
\[ \phi_*\bigg([\M_{g,\mathbf{a}}(\CZr, 0)]^{\vir}_{\C^*}\bigg) = \sum_{i=0}^{\infty} \left( \frac{\lambda}{r} \right)^{g-1 + \frac{1}{r}\sum a_i} \phi_*(c_i(-R\pi_* \mathcal{L}_A)),\]
where $\M_{g,\mathbf{a}}(\CZr, 0)$ denotes the substack of the moduli space of stable maps to $\CZr$ where the monodromy at the $i$th marked point is given by $a_i$.  This follows, for example, from the localization computations in Appendix \ref{localization} for $\P[r,1]$, in the case where the degree $d$ is zero. 

It follows that
\begin{equation}
\label{CZr}
\OmegaCZr_{g,n}(\zeta_{a_1} \otimes \cdots \otimes \zeta_{a_n}) = r^{g-1+\frac{2}{r}\sum a_i} \phi_*\bigg([\M_{g,\mathbf{a}}(\CZr, 0)]^{\vir}_{\C^*}\bigg)\bigg|_{\lambda =\frac{1}{r}}.
\end{equation}
Note that one must be careful in the situation where $a_1 = \cdots = a_n = 0$, since in this case, $\M_{g, \mathbf{a}}(\CZr, 0)$ is noncompact, and the virtual cycle should be understood as defined via the localization formula.
\end{remark} 

\begin{lemma}
  \label{CohFTs}
  Both $\OmegaZvr_{g,n}$ and $\OmegaCZr_{g,n}$ form semisimple
  Cohomological Field Theories with unit $\zeta_0$.
\begin{proof}
  For $\OmegaCZr_{g,n}$, the CohFT property follows from the
  interpretation (\ref{CZr}).
  Indeed, the equivariant Gromov--Witten theory of $\CZr$ forms a
  CohFT under the pairing
  \begin{equation*}
    \eta_{\CZr}(\zeta_i, \zeta_j) = \begin{cases} \frac{1}{\lambda} & \text{ if } i=j=0\\ \frac{1}{r} & \text{ if } 0 \neq i+j\equiv 0 \mod r\\ 0 & \text{ otherwise},\end{cases}
  \end{equation*}
  and the pre-factor $r^{g-1+\frac{2}{r}\sum a_i}$ can easily be shown
  to respect the decomposition properties.
  In general, the proof that both of $\OmegaZvr_{g,n}$ and
  $\OmegaCZr_{g,n}$ form CohFTs follows from the fact that both are
  twisted theories (over $B\Z_r$), which are studied in \cite{Tseng}.
  The CohFT axioms are consequences of the pullback and splitting
  properties satisfied by the K-theory class $R\pi^*\mathcal{L}_A$
  (compare to \cite[Lemma~B.0.9]{Tseng}).

The quantum product in either case can be computed explicitly, since the only contribution to the genus-zero three-point invariants comes in cohomological degree zero.  Thus,
\[\OmegaZvr_{0,3}(\zeta_{a_1} \otimes \zeta_{a_2} \otimes \zeta_{a_3}) = \OmegaCZr_{0,3}(\zeta_{a_1} \otimes \zeta_{a_2} \otimes \zeta_{a_3}) = \begin{cases} 1 & \text{ if } \sum a_i \equiv 0 \mod r,\\ 0 & \text{ otherwise}.\end{cases}\]
It follows that the quantum products are both
\[\zeta_i \ast \zeta_j = \zeta_{i+j \mod r}.\]
This shows that the unit is $\zeta_0$, and moreover, that the ring structure on $V$ is
\[\frac{\C[\zeta_1]}{(\zeta_1^r =1)}.\]
It is easy to see that this ring is semisimple, with idempotents given by
\[\epsilon_i := \frac 1r \sum_{j = 0}^{r - 1} \xi^{ij} \zeta_1^j\]
for $i \in \{0, \ldots, r-1\}$, where $\xi$ is a primitive $r$th root of unity.
\end{proof}
\end{lemma}

It follows from Lemma~\ref{CohFTs} that both $\OmegaZvr_{g,n}$ and $\OmegaCZr_{g,n}$ can be computed in terms of an $R$-matrix action on a TFT.  The TFTs are easy to calculate, since they arise from projecting the CohFT to cohomological degree zero; the result, in either case, is
\[\omega_{g,n}(\zeta_{a_1} \otimes \cdots \otimes \zeta_{a_n}) = \begin{cases} r^g & \text{ if } \sum_{i=1}^n a_i \equiv 0 \mod r,\\ 0 & \text{ otherwise}.\end{cases}\]

\subsection{Computation of $R$-matrices}

Fix the basis $\{ \zeta_0, \ldots, \zeta_{r-1}\}$ for $V$.  We claim that, in this basis, the $R$-matrix associated to the CohFT $\OmegaZvr_{g,n}$ is equal to
\begin{equation}
  \RZvr(z) = \exp
  \begin{pmatrix}
    -\frac{r^2 B_2(0)}2 z && \\ & \ddots & \\ && -\frac{r^2 B_2(\frac{r-1}r)}2 z
  \end{pmatrix} ,
\end{equation}
and that the $R$-matrix associated to the CohFT $\OmegaCZr_{g,n}$ is
\begin{equation}
\label{Rclaim}
\RCZr(z) = \exp
\begin{pmatrix}
  \sum_{d=1}^\infty \frac{B_{d+1}(0)}{d(d+1)} (-r^2z)^d && \\
  & \ddots & \\
  && \sum_{d=1}^\infty \frac{B_{d+1}(\frac{r-1}r)}{d(d+1)} (-r^2z)^d
\end{pmatrix} ,
\end{equation}
where in both cases the matrix inside the exponential is diagonal.\footnote{The fact that these matrices satisfy the symplectic condition $R(z) \cdot R^*(-z) = 1$, where $*$ denotes the adjoint with respect to the pairing, is a straightforward consequence of the identity $B_n(1-x) = (-1)^n B_n(x)$.}

The argument is essentially the same in either case, so we focus on the slightly more complicated situation for $\OmegaCZr_{g,n}$.  By Lemma 2.2 of \cite{MOPPZ}, it suffices to verify that $(\RCZr \cdot \omega)_{g,n}$ agrees with $\OmegaCZr_{g,n}$ when restricted to the open locus $\mathcal{M}_{g,n} \subset \M_{g,n}$.  The only graph contributing to the $R$-matrix action on the open locus is a single vertex with $n$ legs, for which the contribution is
\begin{align}
\label{Tomega}
&(T\omega)_{g,n}((\RCZr)^{-1}(\psi_1) \zeta_{a_1} \otimes \cdots \otimes (\RCZr)^{-1}(\psi_n)\zeta_{a_n})\\
\nonumber=&\sum_{m \geq 0} \frac{1}{m!} p_{m*} \bigg(\omega_{g,n}\big((\RCZr)^{-1}(\psi_1) \zeta_{a_1} \otimes \cdots \otimes (\RCZr)^{-1}(\psi_n)\zeta_{a_n}\\
\nonumber&\hspace{4cm}\otimes T(\psi_{n+1}) \otimes \cdots \otimes T(\psi_{n+m})\big)\bigg).
\end{align}
Here,
\begin{align*}
T(z) &= z \mathbf{1} - z (\RCZr)^{-1}(z) \mathbf{1}\\
&=z\left( 1 - \exp\left(-\sum_{d=1}^{\infty}\frac{ B_{d+1}(0)}{d(d+1)}(-r^2z)^d\right)\right)\zeta_0.
\end{align*}

Using the definition of $\omega_{g,n}$ and applying Lemma 2.3 of \cite{PixtonThesis} to the power series
\[X(t) = 1 - \exp\left(-\sum_{d=1}^{\infty} \frac{ B_{d+1}(0)}{d(d+1)} (-r^2t)^d\right),\]
we can re-write (\ref{Tomega}) as
\[r^{g}\exp\left( \sum_{d=1}^{\infty}(-1)^d\left( \frac{r^{2d} B_{d+1}(0)}{d(d+1)} \kappa_d - \sum_{j=1}^n \frac{r^{2d} B_{d+1}(\frac{a_j}{r})}{d(d+1)} \psi_j^d\right)\right).\]
The classes $\kappa_d$ and $\psi_j$ are pulled back under the degree-$r^{2g-1}$ map $\phi: \mathcal{M}^{0/r}_{g,A} \rightarrow \mathcal{M}_{g,n}$.  Thus, the above is equal to 
\[\frac{1}{r^{g-1}}\phi_*\exp\left( \sum_{d=1}^{\infty}(-1)^d\left( \frac{r^{2d} B_{d+1}(0)}{d(d+1)} \kappa_d - \sum_{j=1}^n \frac{r^{2d} B_{d+1}(\frac{a_j}{r})}{d(d+1)} \psi_j^d\right)\right),\]
where we use the same notation for the $\kappa$ and $\psi$ classes on $\mathcal{M}_{g,n}$ as for their pullbacks to $\mathcal{M}^{0/r}_{g,A}$.  Now, by Chiodo's formula, the above coincides precisely with the restriction of $\OmegaCZr_{g,n}(\zeta_{a_1} \otimes \cdots \otimes \zeta_{a_n})$ to $\mathcal{M}_{g,n}$.

\subsection{Proof of Theorem~\ref{main}}

We can now conclude the proof of the main theorem.

\begin{proof}[Proof of Lemma~\ref{claim2} and Theorem~\ref{main}]
When $A$ has exactly one negative entry, we have reduced the claim to proving Lemma~\ref{claim2}, or in other words that
\[\left. \frac{1}{r^g} (\RZvr \cdot \omega)_{g,n}(\zeta_{a_1'} \otimes \cdots \otimes \zeta_{a_n'})\right|_{r=0} \hspace{-0.4cm}= \left. \frac{1}{r^g} (\RCZr \cdot \omega)_{g,n}(\zeta_{a_1'} \otimes \cdots \otimes \zeta_{a_n'})\right|_{r=0}.\]
This follows from Lemma~\ref{polynomiality}, using the fact that the
lowest-order terms in $r$ of the two $R$-matrices agree.

Thus, the theorem is proved in the case where exactly one $a_i$ is negative.  Since, $\Omega_{g,A}$ is polynomial in $A$ by Lemma \ref{polynomiality}, this implies the result in general as long as $n > 0$.

If $n=0$, the initial step of replacing $A$ by $A'$ is no longer valid, but the above nevertheless implies that
\[\Omega_{g,\emptyset} = \left. \frac{1}{r^{2g-1}} \phi_*\big(c_{(r^2)}(-R\pi_*\mathcal{L}_{\emptyset})\big) \right|_{r=0}.\]
In this case, $R^0\pi_*\mathcal{L}_{\emptyset}$ is a trivial line bundle, while $R^1\pi_*\mathcal{L}_{\emptyset}$ is the pullback under $\phi$ of the Hodge bundle $\mathbb{E}$ on $\M_g$.  Thus, the vanishing of $\Omega_{g,\emptyset}$ in degrees past $g$ follows from the fact that $\mathbb{E}$ is a rank-$g$ vector bundle.
\end{proof}

The computation of $R$-matrices also reveals why the geometric reformulation of $\Omega_{g,A}$ as the constant term of $\Zvr{g}{A}$ matches Pixton's original presentation as the constant term of $\Conjr{g}{A}$.

\begin{lemma}
  \label{PixtonZvonkine}
The two definitions of $\Conj{g}{A}$ described in Sections \ref{conjecture} and \ref{geom} agree:
\[\Conjr{g}{A}\bigg|_{r=0} = \Zvr{g}{A}\bigg|_{r=0}.\]

\begin{proof}
First, we note that both sides are unaffected if $A$ is replaced by its reduction $A''$ modulo $r$; we have already seen this for $\Zvr{g}{A}$, while for $\Conjr{g}{A}$, it follows easily from the definition.

After this replacement, $r^{-g} \Zvr{g}{A''}$ is a semisimple CohFT, and the formula for it via the $R$-matrix action
  exactly agrees with the formula for $\Conjr{g}{A''}$, except that the modifications
  \begin{align*}
    e^{-\frac 12 r^2 B_2(0) \kappa_1}
    &\rightsquigarrow 1, \\
    e^{\frac 12 r^2 B_2\left(\frac{a_i''}r\right)\psi_i}
    &\rightsquigarrow e^{\frac 12 (a_i'')^2\psi_i}, \\
    \frac{1 - e^{-\frac 12 r^2 \left(B_2\left(\frac{w(h)}r\right)\psi_h + B_2\left(\frac{w(h')}r\right)\psi_{h'}\right)}}{\psi_h + \psi_{h'}}
    &\rightsquigarrow \frac{1 - e^{-\frac 12 w(h)w(h')(\psi_h + \psi_{h'})}}{\psi_h + \psi_{h'}},
  \end{align*}
  need to be done for the vertex, leg and edge factors, respectively.
  Now note that
  \begin{equation*}
    B_2(x) = x^2 - x + \frac 16.
  \end{equation*}
  Hence, the first two modifications amount to a multiplication by
  \begin{equation*}
    e^{\frac{r^2}6 \kappa_1} \prod_{i = 1}^n e^{\frac 12\left(ra_i'' - \frac 16 r^2\right)\psi_i},
  \end{equation*}
  which leaves constant terms in $r$ invariant. The third modification
  also does not affect constant-in-$r$-terms, since
  \begin{multline*}
    r^2 B_2\left(\frac{w(h)}r\right) = r^2 B_2\left(\frac{w(h')}r\right) = (w(h))^2 + rw(h) + \frac{r^2}6 \\
    \equiv (w(h))^2 \equiv -w(h)w(h') \pmod{r}.
  \end{multline*}
\end{proof}
\end{lemma}

\subsection{Relations with powers of the log canonical}

Fix an integer $k$ and a tuple of integers $A = (a_1, \ldots, a_n)$ for which
\[\sum_{i=1}^n a_i \equiv k (2g-2+n) \mod r.\]
As above, let $\M_{g,A}^{k/r}$ denote the moduli space of pointed stable curves with a line bundle $L$ satisfying \eqref{Lk}.

Chiodo's formula extends to these more general moduli spaces with only a small modification.  It reads:
\begin{equation}
\label{Chiodok}
\ch_d(R\pi_{*}\mathcal{L}_{A,k}) = \frac{B_{d+1}(\frac kr)}{(d+1)!}\kappa_d - \sum_{i=1}^n \frac{B_{d+1}(\frac{a_i}{r})}{(d+1)!}\psi_i^d \;+ \hspace{2.5cm}
\end{equation}
\[\hspace{0.75cm}\frac{r}{2} \sum_{\substack{0 \leq l \leq g\\ I \subset [n]}} \frac{B_{d+1}\left(\frac{q_{l,I}}{r}\right)}{(d+1)!} p^*i_{(l,I)*}(\gamma_{d-1}) + \frac{r}{2} \sum_{q=0}^{r-1} \frac{B_{d+1}(\frac{q}{r})}{(d+1)!}j_{(irr,q)*}(\gamma_{d-1}),\]
and the multiplicities $q_{l,I}$ are now determined by the condition
\[q_{l,I} + \sum_{i \in I} a_i \equiv k(2g-2+n) \mod r.\]
Using this, the proof of Theorem~\ref{main} is readily generalized.

\begin{theorem}
\label{thm2}
For any $k$ and any tuple $A$ of integers satisfying $\sum a_i  = k(2g-2+n)$, the component of $\Omega_{g,A,k}$ in degree $d$ vanishes for all $d > g$.
\begin{proof}
  Recall from Lemma~\ref{kpolynomiality} that the class
  $\Omega_{g,A,k}$ is polynomial in $k$. Therefore, it suffices to
  prove the theorem only for $k<0$.  In this case, the argument in the
  proof of Theorem~\ref{main} extends straightforwardly.
  
  Specifically, Lemma~\ref{PixtonZvonkine} again shows that the two definitions of $\Omega_{g,A,k}$ agree, so it suffices to prove the vanishing for the geometrically formulated class.  When exactly
  one $a_i$ is negative, one can also replace $A$ by $A' = (a_1 + r,
  a_2, \ldots, a_n)$, which again makes $-R\pi_*(\mathcal{L}_{A', k})$ a vector bundle of rank $g$ (using that $k<0$ to ensure that Lemma~\ref{vbundle} still holds) but does not affect the lowest-order term
  in $r$ of $\phi_*(e^{r^2 c_1(-R\pi_*\mathcal{L}_{A,k})})$.  From
  here, one proves again that the constant-in-$r$ term of
\begin{equation}
\label{CohFT1}
\frac{1}{r^{2g-1}} \phi_*(e^{r^2 c_1(-R\pi_*\mathcal{L}_{A',k})})
\end{equation}
agrees with that of
\begin{equation}
\label{CohFT2}
\frac{1}{r^{2g-1}} \phi_*(c_{(r^2)}(-R\pi_*\mathcal{L}_{A',k})),
\end{equation}
assuming $r$ is first chosen sufficiently large.  The proof is the same as previously; indeed, after multiplying by $r^g$, both (\ref{CohFT1}) and (\ref{CohFT2}) form CohFTs on the same vector space $V$ with the same pairing $\eta$ as considered previously.  The TFT, on the other hand, is now nonzero only when $\sum_{i=1}^n a_i \equiv k(2g-2+n) \; \mod r$, and the unit is not $\zeta_0$ but $\zeta_k$.  This shifted unit, which appears in the definition of the translation operator $T$, precisely accounts for the modification to Chiodo's formula.  A comparison of the $R$-matrices again completes the proof.
\end{proof}
\end{theorem}

\section{Connection to the $3$-spin relations}
\label{3spin}

Theorem~\ref{main} can be viewed as a collection of tautological relations in $A^*(\M_{g,n})$, which we refer to as the {\it double ramification cycle relations}.  Given that Pixton's $3$-spin relations, described in \cite{Pixton} and proved in \cite{PPZ}, are conjectured to generate all relations in the tautological ring, one would expect the double ramification cycle relations to follow from these.  This is indeed the case, as we explain in this section.

More precisely, what we prove is that the double ramification cycle relations $[\Omega_{g,A}]_d$ in which exactly one of the arguments $a_i$ is negative lie in the ideal of the strata algebra generated by the $3$-spin relations.  These are the $A$ for which the arguments of Sections \ref{proof} and \ref{Rmatrices} apply, and thus, for which the double ramification cycle relations can be understood in terms of the equivariant Gromov--Witten theory of $[\C/\Z_r]$; see Remark \ref{GWtheory}.  By polynomiality of $\Omega_{g,A}$ in $A$ (Lemma~\ref{polynomiality}), the relations for these choices of $A$ are sufficient to derive all of the double ramification cycle relations.

It should be noted that the arguments of this section do not apply to the relations of Theorem~\ref{thm2}.  To address this more general situation, one would need to construct a new variant of moduli spaces of stable maps and study its intersection theory.

\subsection{Strata-valued field theories}

We first recall the definition of the \emph{strata algebra}, following \cite{GPNontautological} and \cite{PixtonThesis}.  Let $\Gamma$ be a stable graph of genus $g$ with $n$ legs, let
\[\ \M_{\Gamma}:= \prod_{\text{vertices }v} \M_{g(v),\text{val}(v)},\]
and let $\iota_{\Gamma}: \M_{\Gamma} \rightarrow \M_{g,n}$ be the gluing morphism as in \eqref{iotagamma}.  A \emph{basic class} on $\M_{\Gamma}$ is defined as an expression of the form
\[\gamma:=\prod_{\text{vertices }v} \theta_v,\]
where $\theta_v$ is a monomial in the $\kappa$ and $\psi$ classes on the vertex moduli space $\M_{g(v),\text{val}(v)}$.

The strata algebra $\mathcal{S}_{g,n}$ is generated as a
$\C$-vector space by pairs $[\Gamma, \gamma]$, where $\Gamma$ is a
stable graph and $\gamma$ is a basic class on $\M_{\Gamma}$.
A multiplication rule and a grading
$\mathcal{S}_{g,n} = \bigoplus_{d = 0}^{3g - 3 + n} \mathcal{S}_{g,n}^d$ can be defined so
that the association
\[Q: \mathcal{S}_{g,n} \rightarrow A^*(\M_{g,n})\]
\[[\Gamma, \gamma] \mapsto \iota_{\Gamma*}(\gamma)\]
is a degree-preserving homomorphism of rings.  It was proved by Graber-Pandharipande \cite{GPNontautological} that the classes $Q([\Gamma, \gamma])$ are additive generators of the tautological ring.  Thus, tautological relations can be understood explicitly as elements of the kernel of $Q$.

Generalizing the notion of a CohFT, we define a \emph{strata-valued field theory} as a finite-dimensional vector space $V$ equipped with a nondegenerate pairing $\eta \in V$, a distinguished element $\one \in V$, and a system of homomorphisms
\begin{equation*}
  \Omega_{g,n}: V^{\otimes n} \rightarrow \mathcal{S}_{g,n}
\end{equation*}
for each $g$ and $n$, satisfying the same compatibility axioms as
required for CohFTs; since the $S_n$-action, and the pullbacks under
the gluing and forgetful maps can all be defined at the level of the
strata algebra, these axioms all still make sense.

Via the analogue in cohomology of the homomorphism $Q$, any strata-valued field theory induces a CohFT.  Moreover, for semisimple CohFTs, the graph sum in the Givental-Teleman reconstruction yields a natural lift to a strata-valued field theory.  In particular, the two CohFTs $\OmegaZvr_{g,n}$ and $\OmegaCZr_{g,n}$ considered in Section \ref{Raction} can both be lifted to strata-valued field theories.

As in Section \ref{proof}, fix a collection of integers $A=(a_1, \ldots, a_n)$ whose sum is zero, such that $a_1< 0$ and $a_i \geq 0$ for all $i \geq 2$.  Let $A'$ be as in equation \eqref{A'}.  Lifting $\OmegaCZr_{g,n}$ to a strata-valued field theory as explained above, we define
\[\Omega_{g,A} := \left.\frac{1}{r^g} \OmegaCZr_{g,n}(\zeta_{a_1'} \otimes \cdots \otimes \zeta_{a_n'})\right|_{r=0} \in \mathcal{S}_{g,n}.\]
(We have abused notation somewhat by using the same symbol to denote $\Omega_{g,A}$ and its lift to the strata algebra.)  Throughout this subsection, then, the double ramification cycle relations are viewed as the statement that
\begin{equation*}
  [\Omega_{g,A}]_d \in \ker(Q)
\end{equation*}
for all $A$ as above and all $d > g$.

The coefficients of $[\Omega_{g,A}]_d$, as explained below, are the coefficients of negative powers of the equivariant parameter in the equivariant Gromov--Witten
theory of $[\C/\Z_r]$.
On the other hand, the general machinery of \cite{JandaThesis}, which
also captures the 3-spin relations, is related to the existence of
non-semisimple shifts.
While it is possible to shift the Gromov--Witten theory of $[\C/\Z_r]$
to non-semisimple points, the shifted theory no longer admits a non-equivariant limit, and thus it is not clear how to relate the resulting
relations to the double ramification cycle relations.
As a substitute for $[\C/\Z_r]$, we study the related orbifold
projective line, which has the crucial property that it is always
well-defined non-equivariantly.

\subsection{Equivariant orbifold projective line}
\label{StratFTs}

Let $X = \P[r,1]$ denote an orbifold projective line, with one orbifold point of isotropy $\Z_r$ located at $0$.  More explicitly, $X$ can be expressed as a weighted projective space
\[X  = \frac{\C^2 \setminus \{0\}}{\C^*}\]
in which $\C^*$ acts by $\sigma \cdot (x,y) = (\sigma^r x, \sigma y)$.  Let $\C^*$ act on $X$ by $t \cdot [x,y] = [x, t y]$, and let $\lambda$ denote the equivariant parameter.

One can encode the equivariant orbifold Gromov--Witten theory of $X$ in
a CohFT on the vector space $H^*_{CR}(X)$ depending on $\lambda$, a
Novikov variable $q$, and a formal coordinate $\t \in H^*_{CR}(X)$ as
follows. 
For any $v_1, \ldots, v_n \in H^*_{CR}(X)$ and any $g,n$ such that
$2g-2+n > 0$, define
\begin{multline}
  \label{Omegat}
  \Omega^{\t}_{g,n}(v_1 \otimes \ldots \otimes v_n):= \\
  \sum_{d,m \geq 0} \frac{q^d}{m!} (p_{d,m})_*\left(\prod_{i=1}^n
    \ev_i^*(v_i) \cap \prod_{i=n+1}^{n+m} \ev_i^*(\t)\cap
    [\M_{g,n+m}(X,d)]^{\vir}_{\C^*}\right),
\end{multline}
where $p_{d,m}: \M_{g,n+m}(X,d) \rightarrow \M_{g,n}$ is the forgetful
map.
Note that $\Omega^\t_{g, n}$ is obtained from
$\Omega^{\mathbf 0}_{g, n}$ by the shift $\t$.

There are two natural lifts of $\Omega^{\t}_{g,n}$ to a strata-valued
field theory.
The first of these, which we denote by $\Omega_{g,n}^{\text{rec},\t}$,
is given by Givental-Teleman reconstruction.
In order to apply reconstruction, we first must verify generic
semisimplicity, and in order to make sense of generic semisimplicity,
we need to ensure that the infinite sums in \eqref{Omegat} converge:

\begin{lemma}
\label{semisimple}
The CohFT $\Omega^{\t}_{g,n}$ is regular in all of its parameters
$\lambda$, $q$ and $\mathbf t$.
For any fixed $(\lambda,q) \neq (0,0)$, $\Omega^{\t}_{g,n}$
is semisimple for generic $\t$.
\begin{proof}
  The regularity of $\Omega^{\t}$ is a consequence of the grading by
  cohomological degree: a natural homogeneous basis for the
  equivariant Chen-Ruan cohomology is given by the unit $\one$, the
  equivariant hyperplane class $h = [0]$, and the generators
  $\zeta_0, \ldots, \zeta_0^{r-1}$ of the twisted sectors.
  We can write
  \begin{equation*}
    \t = t_0\one + t_{1/r}\zeta_0 + \cdots + t_{(r-1)/r}\zeta_0^{r-1} + t_1h.
  \end{equation*}
With the grading
  $\deg(\lambda) = 1$, $\deg(qe^{t_1}) = 1 + \frac 1r$, and
  $\deg(t_i) = 1 - i$ for $i \neq 1$, the divisor equation implies that for
  fixed $g$ and homogeneous arguments $v_1, \dotsc, v_n$, the class
  $\Omega_{g,n}^{\t}(v_1 \otimes \cdots \otimes v_n)$ is homogeneous.
  Since the degree of each parameter is positive, the class is
  polynomial in its parameters, and regularity follows.

  Semisimplicity is an open condition on $\t$, so it suffices to
  find a single value of $\t$ for which $\Omega_{g,n}^{\t}$ is
  semisimple.
  Set all coordinates of $\t$ except for $t_1$ equal to zero.
  On this line on the Frobenius manifold, the CohFT is semisimple away
  from the vanishing locus of the discriminant
  \begin{equation*}
    d_{\lambda, q}(t_1) = (-1)^{\binom{r + 1}{2}}\left(\frac{(r + 1)^{r + 1}}{r^r} (qe^{t_1})^r - \frac 1r \lambda^{r + 1}\right)
  \end{equation*}
  of its defining polynomial (see Appendix~\ref{setup}).
  As long as $(\lambda, q) \neq (0,0)$, there exists a choice of $t_1$
  for which $d_{\lambda,q}(t_1) \neq 0$, and hence this choice makes
  $\Omega^{(0, \ldots, 0, t_1)}$ semisimple.
\end{proof}
\end{lemma}

For any $\t$ for which $\Omega_{g,n}^{\t}$ is semisimple,
$\Omega_{g,n}^{\t}$ can be expressed as a graph sum via the action of
an $R$-matrix on a TFT.
As shown in \cite{JandaThesis}, the coefficients of the $R$-matrix are
regular, except that they may acquire poles at the zero locus of the
discriminant
\begin{equation*}
  d_{\lambda,q}(\t) \in \C[\lambda, t_{1/r}, \dotsc, t_{(r-1)/r}, qe^{t_1}].
\end{equation*}
Thus, $\Omega_{g,n}^{\t}$ defines a strata-valued field theory, which
we denote by
\begin{equation*}
  \Omega_{g,n}^{\text{rec}, \t}(v_1 \otimes \ldots \otimes v_n) \in \mathcal{S}_{g,n} \otimes U,
\end{equation*}
where
\begin{equation*}
  U := \C[\lambda^{\pm 1}, t_{1/r}, \dotsc, t_{(r-1)/r}, qe^{t_1}, (d_{\lambda,q}(\t))^{-1}].
\end{equation*}

The other lift of $\Omega_{g,n}^{\t}$ to a strata-valued field theory, which we denote by $\Omega^{\text{loc},\t}_{g,n}$, follows from localization.  More specifically, just as in the case of ordinary $\P^1$, the localization formula expresses $[\M_{g,n+m}(X,d)]^{\vir}_{\C^*}$ as a sum over decorated graphs, in which vertices indicate components contracted to $0$ or $\infty$ in $X$ and edges indicate noncontracted components.  The contributions of each such graph have been explicitly calculated by Johnson \cite{Johnson} and are recalled in the Appendix.  In particular, the moduli at each vertex $v$ is of the form $\M^{0/r}_{g(v),A(v)}$ for some $g(v) \geq 0$ and some tuple of integers $A(v)$, and the contribution of $v$ to the localization expression for $(p_{d,m})_*([\M_{g,n+m}(X,d)]^{\vir}_{\C^*})$ is of the form
\begin{equation}
\label{vertexcontr}
\sum_{i=0}^{\infty} \left(\frac{\lambda_{j(v)}}{r_{j(v)}}\right)^{g(v) - 1 + \iota(\rho(v)) -i} \phi_*\bigg(c_i(-R\pi_*\mathcal{L}_{A(v)})\bigg),
\end{equation}
where the notation is defined in Appendix \ref{localization}.  By applying Chiodo's formula (which, for vertices contracted to $\infty$, reduces to Mumford's formula for the Chern characters of the Hodge bundle), one obtains a natural lift of \eqref{vertexcontr} to the strata algebra.  Doing this at each vertex of each graph in the localization expression for \eqref{Omegat} yields the definition of
\[\Omega_{g,n}^{\text{loc},\t}(v_1 \otimes \ldots \otimes v_n) \in \mathcal{S}_{g,n}(\lambda)[[q,\t]].\]

\subsection{Nonequivariant limit and the double ramification cycle relations}
\label{noneq}

We focus first on the strata-valued field theory $\Omega_{g,n}^{\text{loc},\t}$.  Consider the basepoint $\t = \mathbf{0}$, and set $q=0$, so that the only graph contributing to the localization formula consists of a single vertex.  We then have the following:

\begin{lemma}
\label{lemma1}
For each $A$ with exactly one negative entry and each $d >g$, the double ramification cycle relation $[\Omega_{g,A}]_d \in \mathcal{S}_{g,n}^d$ lies in the ideal of $\mathcal{S}_{g,n}$ generated by the coefficients of negative powers of $\lambda$ in $\Omega^{\text{loc}, \mathbf{0}}_{g,n}(\zeta_{a_1'} \otimes \ldots \otimes \zeta_{a_n'})|_{q=0}$.

\begin{proof}
Since at least one of the integers $a_i'$ is nonzero, the single vertex in the localization graph must map to $0 \in X$.  The fixed locus associated to this graph is $\M_{g,A'}$, and, by the localization contribution recalled in \eqref{vertexcontr}, we have
\begin{equation}
\label{locsum}
\Omega^{\text{loc}, \mathbf{0}}_{g,n}(\zeta_{a_1'} \otimes \ldots \otimes \zeta_{a_n'})|_{q=0}=\sum_{i=0}^{\infty} \left(\frac{\lambda}{r}\right)^{g-i}\phi_*\bigg(c_i(-R\pi_*\mathcal{L}_{A'})\bigg),
\end{equation}
where the right-hand side is lifted to the strata algebra via Chiodo's formula.  In other words, up to a factor of a power of $r$, $[\Omega_{g,A'}^r]_d$ agrees with the coefficient of $\lambda^{g-d}$ in $\Omega^{\text{loc}, \mathbf{0}}_{g,n}(\zeta_{a_1'} \otimes \ldots \otimes \zeta_{a_n'})|_{q=0}$ as elements of the strata algebra.   After taking $r$ sufficiently large and taking the coefficient of the appropriate power of $r$, the lemma follows.
\end{proof}
\end{lemma}

\subsection{Nonsemisimple limit and the $3$-spin relations}

Lemma~\ref{lemma1} shows that the double ramification cycle relations are equivalent to the existence of the nonequivariant limit of
\[Q\left(\Omega^{\text{loc}, \mathbf{0}}_{g,n}(\zeta_{a_1'} \otimes \ldots \otimes \zeta_{a_n'})|_{q=0}\right) \in A^*(\M_{g,n})(\lambda).\]
The $3$-spin relations, on the other hand, arise via the existence of a different limit: the nonsemisimple limit of the reconstruction graph sum.

More specifically, as discussed above, for fixed
$(\lambda, q) \neq 0$, the strata-valued field theory
$\Omega_{g,n}^{\text{rec},\t}$ aquires singularities at values of $\t$
for which the discriminant $d_{\lambda,q}(\t)$ vanishes, reflecting
the failure of reconstruction at these basepoints.
Yet the original CohFT $\Omega^{\t}_{g,n}$ is regular (even
polynomial) in $\t$.
Thus, the regularity of
$Q(\Omega_{g,n}^{\text{rec},\t}) = \Omega_{g,n}^{\t}$ is equivalent to
the condition that the coefficients of
$Q(\Omega_{g,n}^{\text{rec},\t})$ with poles in $d_{\lambda,q}(\t)$
lie in the kernel of $Q$.

This yields a family of tautological relations, and similar
reasoning produces relations associated to any CohFT for which generic
shifts are semisimple.
Following \cite[Definition 3.3.1]{JandaThesis}, we define the
relations $I^{\Lambda}_{g,n} \subset \mathcal{S}_{g,n}$ associated to
a generically semisimple CohFT $\Lambda$ to be the smallest system of
ideals that is stable under pushforwards via the gluing and forgetful morphisms,
and containing the relations from poles in the discriminant
described above.
In particular, taking $\Lambda$ to be the $3$-spin CohFT described in
\cite{PPZ}, this process yields the ideal of the $3$-spin relations.

Surprisingly, these relations are independent of the particular CohFT $\Lambda$ used to generate them:

\begin{theorem}\cite[Theorem 3.3.6]{JandaThesis}
  \label{mrelations}
  Let $\Lambda$ be a generically semisimple, but not everywhere
  semisimple, CohFT.
  Then $I^{\Lambda}_{g,n} = \mathcal{P}_{g,n}$, where
  $\mathcal{P}_{g,n}$ is the ideal of the $3$-spin relations.
\end{theorem}

Applying Theorem~\ref{mrelations} to the CohFT $\Omega_{g,n}^{\t}$
associated to the equivariant orbifold projective line for any fixed
$(\lambda, q) \neq 0$, we obtain the following:

\begin{lemma}
  \label{recsum}
  The image of $\Omega_{g,n}^{\text{rec},\t}$ in
  $(\mathcal{S}_{g,n}/\mathcal{P}_{g,n}) \otimes U$ is
  regular in the variables $\lambda,q$, and $\t$.
\end{lemma}
\begin{proof}
  Let $\overline\Omega_{g,n}^{\text{rec},\t}$ be the image of
  $\Omega_{g,n}^{\text{rec},\t}$ in
  $(\mathcal{S}_{g,n}/\mathcal{P}_{g,n}) \otimes U$.  For $\t$ in the open locus of semisimple basepoints, $\overline\Omega_{g,n}^{\text{rec},\t}$ is regular in all of its variables, while in the non-semisimple locus, Theorem 3 implies that $\overline\Omega_{g,n}^{\text{rec},\t}$ is regular in $\t$ for fixed $(\lambda,q) \neq (0,0)$.  Therefore, $\overline\Omega_{g, n}^{\text{rec},\t}$ is regular in
  all variables outside of the locus $\{(\lambda,q) = (0,0)\}$.
  Since this locus has codimension two, it follows that
  $\overline\Omega_{g, n}^{\text{rec},\t}$ is regular everywhere. 
\end{proof}

\subsection{Comparison of the strata-valued field theories}

In order to prove that the $3$-spin relations imply the double ramification cycle relations, what remains is to compare the two strata-valued field theories.  If one identifies negative powers in $d_{\lambda,q}(\t)$ with their Taylor expansions, then:

\begin{lemma}
  \label{comparison}
  One has
  $\Omega_{g,n}^{\text{loc},\t} = \Omega_{g,n}^{\text{rec},\t}$.
\end{lemma}

The proof of Lemma~\ref{comparison} is technical, and is relegated to the Appendix.  Assuming it, Theorem~\ref{main2} is now immediate:

\begin{proof}[Proof of Theorem~\ref{main2}]
By Lemma~\ref{lemma1}, it suffices to prove that
\[\Omega_{g,n}^{\text{loc},\mathbf{0}}\bigg|_{q=0} \in \frac{\mathcal{S}_{g,n}}{\mathcal{P}_{g,n}}(\lambda)\]
is regular in $\lambda$, and this follows from Lemmas~\ref{recsum} and \ref{comparison}.
\end{proof}

\appendix

\section{Localization on an orbifold projective line}

Let $X=\P[r_0, r_{\infty}]$ be a projective line with an orbifold point of order $r_0$ at $0$ and an orbifold point of order $r_{\infty}$ at $\infty$, on which $\C^*$ acts by $t \cdot [x,y] = [x,ty]$.  The case needed above is $r_0=r$ and $r_{\infty}=1$, but we consider the more general setting here as it may be of independent interest.

The goal of this Appendix is to prove Lemma 6.5--- that is, that the localization computation of the CohFT associated to the equivariant Gromov--Witten theory of $X$ produces the same result as the computation via Givental-Teleman reconstruction, not only on the level of cohomology but in the strata algebra.  In addition to finishing the proof of Theorem~\ref{main2}, this also immediately implies that the Givental-Teleman classification, which in general is only known to hold in cohomology, also holds in the Chow ring in this case.

Our strategy is closely modeled on \cite{CGT}.

\subsection{Classical and quantum cohomology}
\label{setup}

The torus action on $X$ has two fixed points, $0$ and $\infty$.  We let $\lambda$ denote the equivariant parameter.  The equivariant Chen-Ruan cohomology ring of $X$ is isomorphic to $H_0 \oplus H_{\infty}$, where
\begin{equation*}
  H_0 = \C[\zeta_0]/(\zeta_0^{r_0} - \lambda/r_0), \; \;
  H_{\infty} = \C[\zeta_\infty]/(\zeta_\infty^{r_\infty} -  \lambda/r_\infty).
\end{equation*}
Here, $\zeta_i, \ldots, \zeta_i^{r_i-1}$ for $i \in \{0, \infty\}$ are
the generators of the twisted sectors, and the untwisted sector is
generated by the classes $\phi_0 := [0]/\lambda$ and $\phi_\infty :=
-[\infty]/\lambda$, where $[i]$ are the equivariant classes of the
fixed points.  We denote by $\one = \phi_0 + \phi_\infty$ the identity and by $h=[0]$ the equivariant hyperplane class.

Denote by
\begin{equation*}
  \t = \sum_{i = 1}^{r_0 - 1} t_{i/r_0} \zeta_0^i + \sum_{i = 1}^{r_\infty - 1} t_{i/r_\infty} \zeta_\infty^i + t_0 \one + t_1 h
\end{equation*}
a formal point of $H^*_{CR,\C^*}(X)$.
The equivariant quantum cohomology ring of $X$, viewed as a
deformation of the usual Chen-Ruan cohomology ring parameterized by
$\t$ and a Novikov variable $q$, is semisimple.
In the case where $t_i=0$ for $i\neq 1$ (the small quantum
cohomology), it has been given an explicit description by
Milanov-Tseng \cite{MT}.
Specifically, let $f(x)$ be the Landau-Ginzburg mirror polynomial:
\begin{equation*}
  f(x) = e^{r_0 x} + q^{r_\infty} e^{r_\infty(t_1 - x)}+ \lambda(t_1 - x).
\end{equation*}
Then the small equivariant quantum cohomology ring of $X$ is
isomorphic to the ring generated by $e^{\pm x}$ modulo the
derivative of $f$, under the identification:
\begin{align*}
\zeta_0^i &\mapsto e^{ix}, &\zeta_{\infty}^i &\mapsto q^i e^{i(t_1 - x)}, \\
  \one &\mapsto 1,
  &h &\mapsto r_\infty q^{r_\infty} e^{r_\infty(t_1 - x)} + \lambda.
\end{align*}

\subsection{Localization}
\label{localization}

As in Section~\ref{StratFTs}, the equivariant Gromov--Witten theory of $X$ can be encoded in a strata-valued field theory via localization.  To make this precise, first recall that the fixed loci of the torus action on $\M_{g,n}(X,d)$ are indexed by certain decorated graphs $\G$.  The fixed locus associated to $\G$ parameterizes stable maps $f: C \rightarrow X$ for which:
\begin{itemize}
\item Edges of $\G$ correspond to components $C_e$ of $C$ not contracted by $f$.  Such components must be genus-zero Galois covers of $X$ ramified only over $0$ and $\infty$, and we denote by $d(e)$ the degree of the restriction of $f$ to $C_e$.
\item Vertices of $\G$ correspond (except in certain unstable cases) to subcurves $C_v$ of $C$ contracted by $f$, and we denote by $g(v)$ the genus of $C_v$.  Such a component must map to one of the fixed points of $X$, which we specify by $j(v) \in \{0,\infty\}$.
\item Legs of $\G$ correspond to marked points, and we denote by $\rho(l)$ the twisted sector in $X$ to which $f$ maps the marked point.
\end{itemize}
Let $h(v)$ denote the number of half-edges incident to a vertex $v$, and let $n(v)$ denote the number of legs.  There are three exceptional cases,
\[(g(v),h(v),n(v)) \in \{(0,1,0), (0,1,1), (0,2,0)\},\]
in which a vertex corresponds not to a contracted subcurve but to a single point of $C$.  In these situations, $v$ is referred to as {\it unstable}; otherwise, $v$ is {\it stable}.

Given a stable vertex $v$, the decorations on $\G$ determine the monodromy of the map at all special points of $C_v$ (see \cite[Lemma II.12]{Johnson}), and we encode these in a tuple
\[\rho(v) \in \{0, 1, \ldots, r_{j(v)}-1\}^{h(v) + n(v)}.\]
Note that, since orbifold maps $C \rightarrow B\Z_{r_{j(v)}}$ can be re-interpreted as $r_{j(v)}$-torsion orbifold line bundles, the moduli space $\M_{g(v), \rho(v)}(B\Z_{r_{j(v)}}, 0)$ parameterizing the contribution of $C_v$ is precisely the moduli space $\M^{0/r}_{g(v), A}$ considered above, with $A=\rho(v)$.

Let $V(\G)$ and $E(\G)$ denote the vertex and edge sets of $\G$, and denote
\[F_{\G}:= \prod_{v \in V(\Gamma) \text{ stable}} \M_{g(v), \rho(v)}(B\Z_{r_{j(v)}}, 0).\]
Then there is as canonical family of $\C^*$-fixed stable maps to $X$ over $F_{\G}$, yielding a finite morphism
\[j_{\G}: F_{\G} \rightarrow \M_{g,n}(X,d)\]
onto the fixed locus associated to $\G$.  Thus, applying the virtual localization formula, $[\M_{g,n}(X,d)]^{\vir}_{\C^*}$ can be expressed as a sum over decorated graphs $\G$ of contributions pushed forward from the moduli spaces $F_{\G}$.  Specifically, the calculations of Johnson in \cite{Johnson} show:
\begin{multline}
  \label{eq:locresult}
  [\M_{g,n}(X,d)]^{\vir}_{\C^*} = \\
  \sum_{\G} \frac{(j_{\G})_*}{|\Aut(\G)|} \left( \prod_{e \in E(\G)} C(e) \prod_{\substack{v \in V(\G) \\ \text{stable}}} C(v) \hspace{-8pt} \prod_{\substack{v \in V(\G) \\ (g,h,n) = (0,1,0)}} \hspace{-8pt} (-\psi_v) \prod_{\text{nodes}} \frac{\eta_{e, v}^{-1}}{-\psi - \psi'} \right).
\end{multline}
Here, setting $\lambda_0 := \lambda$ and $\lambda_{\infty} :=
-\lambda$, we denote
\[C(e) := \lambda_0^{-\left\lfloor \frac{d(e)}{r_0}\right\rfloor} \lambda_{\infty}^{-\left\lfloor \frac{d(e)}{r_{\infty}}\right\rfloor} \cdot \frac{d(e)^{\left\lfloor \frac{d(e)}{r_{0}}\right\rfloor +\left\lfloor \frac{d(e)}{r_{\infty}}\right\rfloor-1}}{\left\lfloor \frac{d(e)}{r_{0}}\right\rfloor!\left\lfloor \frac{d(e)}{r_{\infty}}\right\rfloor!}\]
and
\[ C(v) := \sum_{i=0}^\infty \left(\frac{\lambda_{j(v)}}{r_{j(v)}}\right)^{g(v) -1 + \iota(\rho(v)) - i} c_i(-R\pi_*\mathcal L),\]
where $\pi: \mathcal{C} \rightarrow \M_{g(v),\rho(v)} \left(B\Z_{r_{j(v)}}\right)$ is the universal curve, $\mathcal{L}$ is the universal $r_{j(v)}$-torsion line bundle, and
\[\iota(\rho(v)) := \sum_{a \in \rho(v)} \frac{a}{r_{j(v)}}.\]
In the third product, $\psi_v$ denotes
the equivariant cotangent line class of the coarse underlying curve--- i.e., $-\psi_v = \lambda_{j(v)}/d(e)$, where $e$ is the unique adjacent
edge. The last product is over nodes forced on the source curve by
$\G$, and $\psi$ and $\psi'$ are the equivariant
cotangent line classes of the coarse underlying curves joined by the
node, while
\begin{equation*}
  \eta_{e,v}^{-1}:=
  \begin{cases}
    \lambda_{j(v)} & \text{ if } d(e) \equiv 0 \mod r_{j(v)} \\
    r_{j(v)} & \text{ if } d(e) \not \equiv 0 \mod r_{j(v)}.
  \end{cases}
\end{equation*}

\subsection{The localization strata-valued field theory}
\label{notation}

Equipped with the results of Section~\ref{localization}, we are ready to more carefully define the strata-valued field theory $\Omega_{g,n}^{\text{loc},\t}$.

First, for $j \in \{0,\infty\}$ and $\rho_1, \ldots, \rho_n \in \{0,1,\ldots, r_j-1\}$, we define
\begin{equation}
 \label{ChFT}
  \Omega^{\Ch, j}_{g, n}(\zeta_j^{\rho_1} \otimes \dotsc \otimes \zeta_j^{\rho_n}) = \sum_{i = 0}^\infty \left(\frac{\lambda_j}{r_j}\right)^{g - 1 + \frac{\sum_{k = 1}^n \rho_k}{r_j}} \phi_*(c_i(-R\pi_*\mathcal L)),
\end{equation}
where
$\phi: \M_{g,(\rho_1, \ldots, \rho_n)}(B\Z_{r_j},0) \rightarrow
\M_{g,n}$ forgets the line bundle and the orbifold structure.
Using Chiodo's formula, one can view \eqref{ChFT} as an element of the
strata algebra, so extending multilinearly, it defines a strata-valued
field theory on $H_j$.
Define $\Omega^{\Ch} = \Omega^{\Ch,0} \oplus \Omega^{\Ch,\infty}$, a
strata-valued field theory on $H^*_{CR,\C^*}(X)$.

A CohFT $\Omega^{\t}_{g,n}$ encoding the equivariant Gromov--Witten
theory of $X$ can be defined by the same formula as in \eqref{Omegat}. 
Via \eqref{eq:locresult}, one can express $\Omega_{g,n}^{\t}$ in terms
of $\Omega^\Ch_{g,n}$, and this defines the lift
$\Omega^{\text{loc},\t}_{g,n}$ of $\Omega^{\t}_{g,n}$ to a
strata-valued field theory.

Toward proving the equality of the localization and reconstruction strata-valued field theories, we first observe that for each localization graph $\G$, there exists a dual graph $s(\G)$ recording the topological type of the stabilization of a generic source curve in the fixed locus associated to $\G$.  Thus, we can write
\[\Omega_{g,n}^{\text{loc},\t}(v_1 \otimes \cdots \otimes v_n) = \sum_{\Gamma \in G_{g,n}}  \frac{1}{|\text{Aut}(\Gamma)|}\sum_{\G \; | \; s(\G) = \Gamma} \text{Cont}_{\G},\]
where $\text{Cont}_{\G}$ is the contribution of $\G$ coming from \eqref{eq:locresult}.

For a fixed dual graph $\Gamma$, the localization graphs $\G$ with
$s(\G) = \Gamma$ are obtained by attaching an arbitrary number of
additional trees of rational curves at each vertex, replacing each leg
by a (possibly empty) tree containing the corresponding marked point,
and replacing each edge with a (possibly empty) tree of rational
curves.
After forming generating series for the contributions of each type of
tree, we can rewrite the contribution of a dual graph $\Gamma$ to
$\Omega_{g,n}^{\text{loc},\t}(v_1 \otimes \cdots \otimes v_n)$ (up to the factor of $|\Aut(\Gamma)|^{-1}$) as a
contraction of strata-valued tensors as follows:
\begin{itemize}
\item at each vertex of $\Gamma$, place the tensor
  $(\epsilon^{\text{pre}} \Omega^{\Ch})_{g(v), \text{val}(v)}$, where
  $\epsilon^{\text{pre}}(z)$ is the generating series of contributions
  of additional trees of rational curves to the localization formula, acting by translation on $\Omega^{\Ch}$;
\item at each leg $l$ of $\Gamma$, place $T^{\text{pre}}(\psi_l) v_l$,
  where $T^{\text{pre}}(z) v$ is the generating series of
  contributions of a tree of rational curves containing a marking with
  a $v$-insertion;
\item at each edge $e = (h, h')$ of $\Gamma$, place
  $E^{\text{pre}}(\psi_h, \psi_{h'})$, where $E^{\text{pre}}(z, w)$ is
  the generating series of contributions of a tree of rational curves
  connecting two stable vertices.
\end{itemize}

The strata-valued tensor $\Omega^{\Ch}_{g(v), \text{val}(v)}$ appearing at each vertex $v$, defined via Chiodo's formula, can be viewed as the result of the action of an $R$-matrix $R_{\Ch}$ on a topological field theory $\omega^{\Ch}$.  Plugging this expression into the above, one gets a sum
\begin{equation}
\label{eq:locconts}
\Omega_{g,n}^{\text{loc},\t}(v_1 \otimes \cdots \otimes v_n) = \sum_{\Gamma \in G_{g,n}} \sum_{\substack{v \in V(\Gamma)\\ \Gamma_v \in G_{g(v), \text{val}(v)}}} \frac{1}{|\text{Aut}(\Gamma')|} \text{Cont}_{\Gamma'},
\end{equation}
where $\Gamma_v$ is a choice of dual graph at the vertex $v$ and $\Gamma'$ is obtained from $\Gamma$ by replacing each $v$ by $\Gamma_v$.  Specifically, $\text{Cont}_{\Gamma'}$ is a contraction of strata-valued tensors as follows:
\begin{itemize}
\item at each vertex of $\Gamma$, place the tensor
  $(\epsilon \omega^\Ch)_{g(v), \text{val}(v)}$, where
  \begin{equation*}
    \epsilon(z) := R_\Ch^{-1}(z) \epsilon^{\text{pre}}(z) + z\one - zR_\Ch^{-1}(z) \one;
  \end{equation*}
\item at each leg $l$ of $\Gamma$, place $T(\psi_l) v_l$,
  where
  \begin{equation*}
    T(z) := R_\Ch^{-1}(z) T^{\text{pre}}(z);
  \end{equation*}
\item at each edge $e = (h, h')$ of $\Gamma$, place
  $E(\psi_h, \psi_{h'})$ where
  \begin{equation*}
    E(z, w) := (R_\Ch^{-1}(z) \otimes R_\Ch^{-1}(w)) E^{\text{pre}}(z, w) + \frac{\eta^{-1} - (R_\Ch^{-1}(z) \otimes R_\Ch^{-1}(w))\eta^{-1}}{z + w}.
  \end{equation*}
\end{itemize}

One important difference, at this point, between \eqref{eq:locconts} and the expression for $\Omega_{g,n}^{\text{rec},\t}$ is that $\epsilon(z)$ is not a multiple of $z^2$, unlike the series $z\one - z R^{-1}(z)\one$ appearing in the $R$-matrix action.  To correct for this difference, it is necessary to modify $\epsilon$ via the series
\begin{equation*}
  u := \sum_{n = 1}^\infty \int_{\M_{0, 2 + n}} \frac 1{n!} \omega^\Ch_{0, 2 + n}\left(\eta^{-1} \otimes \one \otimes \epsilon(\psi)^{\otimes n}\right) \in \mathcal H,
\end{equation*}
in which we use the shorthand $\mathcal{H} := H^*_{CR}(X)(\lambda)[[q,\t]]$.  We observe that, by an application of the string equation and the genus-zero topological recursion relations,
\begin{equation*}
  [e^{u/z} (\epsilon(z) - u)]_+ \in z \mathcal H[[z]],
\end{equation*}
where $[\; \cdot \;]_+$ denotes the truncation to non-negative powers of $z$.

More precisely, for each graph $\Gamma'$ as in \eqref{eq:locconts} and each vertex $v$ of $\Gamma'$, the contribution of $v$ to $\Omega^{\text{loc},\t}_{g,n}(v_1 \otimes \cdots \otimes v_n)$ is
\begin{equation*}
  \sum_{m = 0}^\infty \frac 1{m!} \pi_{m*} \omega^\Ch_{g(v), \text{val}(v) + m}(f_1(\psi_1) \otimes \dotsb \otimes f_{\text{val}(v)}(\psi_{\text{val}(v)}) \otimes \epsilon(\psi)^{\otimes m}),
\end{equation*}
by the definition of the translation action; here, the series $f_i(z)$ stands either for a series $T(z)$ or
``half'' of a series $E(z, w)$.
We split the shift by $\epsilon(z)$ into a shift by $\epsilon(z) - u$
and a shift by $u$:
\begin{equation*}
  \sum_{m, l = 0}^\infty \frac {\pi_{m*} \pi_{l*}}{m!l!}  \omega^\Ch_{g(v), \text{val}(v) + m + l}(f_1(\psi_1) \otimes \dotsb \otimes f_{\text{val}(v)}(\psi_{\text{val}(v)}) \otimes (\epsilon(\psi) - u)^{\otimes m} \otimes u^{\otimes l}).
\end{equation*}
Re-expressing the first $\text{val}(v) + m$ cotangent line classes in terms
of those pulled back via $\pi_{l*}$, we can rewrite the local
contribution of $v$ to $\Omega^{\text{loc},\t}_{g,n}(v_1 \otimes \cdots \otimes v_n)$ as
\begin{multline*}
  \sum_{m, l = 0}^\infty \frac{\pi_{m*} \pi_{l*}}{m!l!} \omega^\Ch_{g(v), \text{val}(v) + m + l}\left(\bigotimes_{i = 1}^n [e^{u/\widetilde\psi_i} f_i(\widetilde\psi_i)]_+ \otimes [e^{u/\widetilde\psi}(\epsilon(\widetilde\psi) - u)]_+^{\otimes m} \otimes u^{\otimes l}\right) \\
  = \sum_{m = 0}^\infty \frac{\pi_{m*}}{m!} \omega^\Ch_{g(v), \text{val}(v) + m}\left(\bigotimes_{i = 1}^n [e^{u/\psi_i} f_i(\psi_i)]_+ \otimes [e^{u/\psi}(\epsilon(\psi) - u)]_+^{\otimes m}\right),
\end{multline*}
in which $\widetilde{\psi}_i$ are the pulled-back $\psi$-classes.

Define a new topological field theory $\widetilde\omega$ by
\begin{equation*}
  \widetilde\omega_{g, n}(v_1 \otimes \dotsb \otimes v_n) := \sum_{m = 0}^\infty \frac {\pi_{m*}}{m!} \omega^\Ch_{g, n+m}(T_0 v_1 \otimes \dotsb \otimes T_0 v_n \otimes (\epsilon_1\psi)^{\otimes m}),
\end{equation*}
where $T_0$ is the $z^0$-coefficient of $e^{u/z} T(z)$ and $\epsilon_1$ is the
$z^1$-coefficient of $e^{u/z} \epsilon(z)$. Then, up to the factor of $|\Aut(\Gamma')|^{-1}$, the
contribution of $\Gamma'$ to $\Omega^{\t, \text{loc}}_{g,n}$ is equal to the following contraction of tensors:
\begin{itemize}
\item at each vertex of $\Gamma$, place the tensor
  $(\widetilde\epsilon \widetilde\omega)_{g(v), \text{val}(v)}$, where
  \begin{equation*}
    \widetilde\epsilon(z) := T_0^{-1}([e^{u/z}(\epsilon(z) - u)]_+ - \epsilon_1 z) \in z^2 \mathcal H[[z]];
  \end{equation*}
\item at each leg $l$ of $\Gamma$, place $\widetilde T(\psi_l) v_l$,
  where
  \begin{equation*}
    \widetilde T(z) := T_0^{-1} [e^{u/z} T(z)]_+;
  \end{equation*}
\item at each edge $e = (h, h')$ of $\Gamma$, place
  $E^{\text{pre}}(\psi_h, \psi_{h'})$, where
  \begin{equation*}
    \widetilde E(z, w) := (T_0^{-1} \otimes T_0^{-1}) [(e^{u/z} \otimes e^{u/w}) E(z, w)]_+.
  \end{equation*}
\end{itemize}
This now has exactly the shape of the definition of $\Omega^{\text{rec},\t}$, and all that remains is to match the ingredients $\widetilde\omega$, $\widetilde\epsilon$, $\widetilde T$ and
$\widetilde E$ with those appearing in the $R$-matrix action.

\subsection{Proof of Lemma \ref{comparison}}

The fact that $\widetilde\omega$ is the same as the topological field theory underlying $\Omega^{\text{rec}, \t}$ is immediate, since both are given by the degree-zero part of the same cohomology class $\Omega^{\t}_{g,n}(v_1 \otimes \cdots \otimes v_n)$.  To see that the series $\widetilde\epsilon$, $\widetilde T$ and
$\widetilde E$ are correct, we also use that $\Omega^{\text{loc},\t}$
agrees with $\Omega^{\text{rec},\t}$ after passing to cohomology.  This implies, by \cite[Lemma~2.2]{MOPPZ}, that $\widetilde T = R^{-1}$, where
$R$ is the $R$-matrix for $\Omega^\t$.

To show that $\widetilde\epsilon(z) = z(\one - R^{-1}(z)\one)$, consider the class
\begin{equation*}
  \Omega^\t_{g, 1}(v)|_{\mathcal{M}_{g, 1}},
\end{equation*}
where $v$ is an idempotent for the quantum product on $H^*_{CR}(X)$.  By Harer stability and \cite{PixtonThesis}, the tautological ring of
$\mathcal{M}_{g, 1}$ in degree less than $g/3$ has an additive basis given by
elements of the form
\begin{equation*}
  \psi_1^a \pi_{m*}(\psi_2^{p_1 + 1} \dotsb \psi_{m + 1}^{p_m + 1}),
\end{equation*}
where $m \geq 0$ and $p_1, \dotsc, p_m \ge 1$.
For any $d \geq 0$ and any idempotent vector $v$, we can choose
$g > 3d$, and then the $z^d$-coefficient in both
$\eta(\widetilde\epsilon(z), v)$ and
$\eta(z(\one - R^{-1}(z)\one), v)$ is equal, up to a common nonzero
constant of $\widetilde\omega_{g, 2}(v, v)$, to the coefficient of
$\pi_{1*}(\psi_2^{d+1})$ in the expression for
$\Omega^{\t}_{g,1}(v)|_{\mathcal{M}_{g,1}}$ in the above basis. 
Thus, we indeed must have
$\widetilde\epsilon(z) = z(\one - R^{-1}(z)\one)$.

Finally, we show that 
\begin{equation}
\label{eq:Es}
\widetilde E(z, w) = \frac{\eta^{-1} - (R^{-1}(z) \otimes R^{-1}(w)) \eta^{-1}}{z + w}.
\end{equation}
Denote the left-hand and right-hand sides of \eqref{eq:Es} by $E_1(z,w)$ and $E_2(z,w)$, respectively, and for $i \in \{1,2\}$, write
\begin{equation*}
  E_i(z, w) =: \sum_{j, k = 0}^\infty E_i^{jk} z^j w^k.
\end{equation*}
We show by induction on $\max(j, k)$ that
$\eta(E_1^{jk}, v_1 \otimes v_2) = \eta(E_2^{jk}, v_1 \otimes v_2)$
for any $j,k$ and any idempotent vectors $v_1$ and $v_2$.
Let us consider, for any $g > 3j$, the coefficient of $\psi_2^j$ in
\begin{equation}
\label{eq:Mg2}
  \left(\pi_{(k+1)*} \Omega^\t_{g, k + 3}(v_1 \otimes v_2^{\otimes k + 2})\right)\big|_{\mathcal{M}_{g, 2}}.
\end{equation}
If one expresses $\Omega^{\t}_{g,k+3}(v_1  \otimes v_2^{\otimes k+2})$ as a dual graph sum (with either $E_1$ or $E_2$ as the edge tensor), then the only graphs contributing to \eqref{eq:Mg2} are those of rational tails type.  By the induction hypothesis, the coefficient of $\psi_2^j$ in the contribution of such a graph is independent of whether the edge tensor $E_1$ or $E_2$ is used, except possibly in the case of the graph with a genus-$g$ vertex connected by a single edge to a genus-zero vertex containing
all except the first marking.  The coefficient of $\psi_2^j$ in the contribution of this graph is
\begin{equation*}
  \eta(E_i^{jk}, v_1 \otimes v_2) \widetilde\omega_{g, 2}(v_1^{\otimes 2}) \widetilde\omega_{0, k+3}(v_2^{\otimes k+3}),
\end{equation*}
so we must also have
$\eta(E_1^{jk}, v_1 \otimes v_2) = \eta(E_2^{jk}, v_1 \otimes v_2)$.  This completes the induction step and hence the proof of Lemma~\ref{comparison}.

\bibliographystyle{abbrv}
\bibliography{DRBib}

\begin{thebibliography}{10}

\bibitem{CavalieriDR}
R.~Cavalieri.
\newblock Hurwitz theory and the double ramification cycle.
\newblock {\em Jpn. J. Math.}, 11(2):305--331, 2016.

\bibitem{CMW}
R.~Cavalieri, S.~Marcus, and J.~Wise.
\newblock Polynomial families of tautological classes on {$M_{g,n}^{rt}$}.
\newblock {\em J. Pure Appl. Algebra}, 216(4):950--981, 2012.

\bibitem{Chiodo}
A.~Chiodo.
\newblock Stable twisted curves and their {$r$}-spin structures.
\newblock {\em Ann. Inst. Fourier (Grenoble)}, 58(5):1635--1689, 2008.

\bibitem{ChiodoTowards}
A.~Chiodo.
\newblock Towards an enumerative geometry of the moduli space of twisted curves
  and {$r$}th roots.
\newblock {\em Compos. Math.}, 144(6):1461--1496, 2008.

\bibitem{CR}
A.~Chiodo and Y.~Ruan.
\newblock L{G}/{CY} correspondence: the state space isomorphism.
\newblock {\em Adv. Math.}, 227(6):2157--2188, 2011.

\bibitem{CZ}
A.~Chiodo and D.~Zvonkine.
\newblock Twisted {$r$}-spin potential and {G}ivental's quantization.
\newblock {\em Adv. Theor. Math. Phys.}, 13(5):1335--1369, 2009.

\bibitem{Clader}
E.~Clader.
\newblock Relations on {$\overline{M}_{g,n}$} via orbifold stable maps.
\newblock {\em Proc. Amer. Math. Soc.}, 145(1):11--21, 2017.

\bibitem{CGT}
T.~Coates, A.~Givental, and H.-H. Tseng.
\newblock Virasoro constraints for toric bundles.
\newblock arXiv: 1508.06282, 2015.

\bibitem{DM}
C.~Deninger and J.~Murre.
\newblock Motivic decomposition of abelian schemes and the {F}ourier transform.
\newblock {\em J. Reine Angew. Math.}, 422:201--219, 1991.

\bibitem{ERW}
J.~Ebert and O.~Randal-Williams.
\newblock Stable cohomology of the universal {P}icard varieties and the
  extended mapping class group.
\newblock {\em Doc. Math.}, 17:417--450, 2012.

\bibitem{FP}
C.~Faber and R.~Pandharipande.
\newblock Relative maps and tautological classes.
\newblock {\em J. Eur. Math. Soc. (JEMS)}, 7(1):13--49, 2005.

\bibitem{Givental}
A.~B. Givental.
\newblock Gromov-{W}itten invariants and quantization of quadratic
  {H}amiltonians.
\newblock {\em Mosc. Math. J.}, 1(4):551--568, 645, 2001.
\newblock Dedicated to the memory of I. G. Petrovskii on the occasion of his
  100th anniversary.

\bibitem{GiventalSemisimple}
A.~B. Givental.
\newblock Semisimple {F}robenius structures at higher genus.
\newblock {\em Internat. Math. Res. Notices}, (23):1265--1286, 2001.

\bibitem{GPNontautological}
T.~Graber and R.~Pandharipande.
\newblock Constructions of nontautological classes on moduli spaces of curves.
\newblock {\em Michigan Math. J.}, 51(1):93--109, 2003.

\bibitem{GH}
S.~Grushevsky and K.~Hulek.
\newblock Geometry of theta divisors---a survey.
\newblock In {\em A celebration of algebraic geometry}, volume~18 of {\em Clay
  Math. Proc.}, pages 361--390. Amer. Math. Soc., Providence, RI, 2013.

\bibitem{GZ}
S.~Grushevsky and D.~Zakharov.
\newblock The double ramification cycle and the theta divisor.
\newblock {\em Proc. Amer. Math. Soc.}, 142(12):4053--4064, 2014.

\bibitem{Hain}
R.~Hain.
\newblock Normal functions and the geometry of moduli spaces of curves.
\newblock In {\em Handbook of moduli. {V}ol. {I}}, volume~24 of {\em Adv. Lect.
  Math. (ALM)}, pages 527--578. Int. Press, Somerville, MA, 2013.

\bibitem{Janda}
F.~Janda.
\newblock Comparing tautological relations from the equivariant
  {G}romov-{W}itten theory of projective spaces and spin structures.
\newblock arXiv: 1407.4778, 2014.

\bibitem{JandaThesis}
F.~Janda.
\newblock Relations in the tautological ring and {F}robenius manifolds near the
  discriminant.
\newblock arXiv: 1505.03419, 2015.

\bibitem{JPPZ}
F.~Janda, R.~Pandharipande, A.~Pixton, and D.~Zvonkine.
\newblock Double ramification cycles on the moduli spaces of curves.
\newblock arXiv:1602.04705, 2016.

\bibitem{Johnson}
P.~Johnson.
\newblock Equivariant {GW} theory of stacky curves.
\newblock {\em Comm. Math. Phys.}, 327(2):333--386, 2014.

\bibitem{KM}
M.~Kontsevich and Y.~Manin.
\newblock Gromov-{W}itten classes, quantum cohomology, and enumerative geometry
  [ {MR}1291244 (95i:14049)].
\newblock In {\em Mirror symmetry, {II}}, volume~1 of {\em AMS/IP Stud. Adv.
  Math.}, pages 607--653. Amer. Math. Soc., Providence, RI, 1997.

\bibitem{MW}
S.~Marcus and J.~Wise.
\newblock Stable maps to rational curves and the relative {J}acobian.
\newblock arXiv: 1310.5981, 2013.

\bibitem{MOPPZ}
A.~Marian, D.~Oprea, R.~Pandharipande, A.~Pixton, and D.~Zvonkine.
\newblock The {C}hern character of the {V}erlinde bundle over the moduli space
  of stable curves.
\newblock arXiv: 1311.3028, 2014.

\bibitem{MT}
T.~E. Milanov and H.-H. Tseng.
\newblock Equivariant orbifold structures on the projective line and integrable
  hierarchies.
\newblock {\em Adv. Math.}, 226(1):641--672, 2011.

\bibitem{Morita1}
S.~Morita.
\newblock Families of {J}acobian manifolds and characteristic classes of
  surface bundles. {I}.
\newblock {\em Ann. Inst. Fourier (Grenoble)}, 39(3):777--810, 1989.

\bibitem{Morita2}
S.~Morita.
\newblock Families of {J}acobian manifolds and characteristic classes of
  surface bundles. {II}.
\newblock {\em Math. Proc. Cambridge Philos. Soc.}, 105(1):79--101, 1989.

\bibitem{PPZ}
R.~Pandharipande, A.~Pixton, and D.~Zvonkine.
\newblock Relations on {$\overline{M}_{g,n}$} via {$3$}-spin structures.
\newblock {\em J. Amer. Math. Soc.}, 28(1):279--309, 2015.

\bibitem{Pixton}
A.~Pixton.
\newblock Conjectural relations in the tautological ring of
  $\overline{M}_{g,n}$.
\newblock arXiv: 1207.1918, 2012.

\bibitem{PixtonThesis}
A.~Pixton.
\newblock The tautological ring of the moduli space of curves.
\newblock Thesis (Ph.D.)-- Princeton University, 2013.

\bibitem{PixtonDRNotes}
A.~Pixton.
\newblock Double ramification cycles and tautological relations on
  $\overline{M}_{g,n}$, 2014.

\bibitem{PixtonForthcoming}
A.~Pixton.
\newblock On combinatorial properties of the explicit expression for double
  ramification cycles (forthcoming), 2015.

\bibitem{RW}
O.~Randal-Williams.
\newblock Relations among tautological classes revisited.
\newblock {\em Adv. Math.}, 231(3-4):1773--1785, 2012.

\bibitem{SSZ}
S.~Shadrin, L.~Spitz, and D.~Zvonkine.
\newblock Equivalence of {ELSV} and {B}ouchard-{M}ari\~no conjectures for
  {$r$}-spin {H}urwitz numbers.
\newblock {\em Math. Ann.}, 361(3-4):611--645, 2015.

\bibitem{Teleman}
C.~Teleman.
\newblock The structure of 2{D} semi-simple field theories.
\newblock {\em Invent. Math.}, 188(3):525--588, 2012.

\bibitem{Tseng}
H.-H. Tseng.
\newblock Orbifold quantum {R}iemann-{R}och, {L}efschetz and {S}erre.
\newblock {\em Geom. Topol.}, 14(1):1--81, 2010.

\bibitem{Voisin}
C.~Voisin.
\newblock Chow rings and decomposition theorems for families of {$K3$} surfaces
  and {C}alabi-{Y}au hypersurfaces.
\newblock {\em Geom. Topol.}, 16(1):433--473, 2012.

\end{thebibliography}

\end{document}